\documentclass[a4paper, 11pt,]{article}
\usepackage{amsfonts}
\usepackage{mathrsfs}
\usepackage{graphics}

\usepackage{amsmath,amsthm}
\usepackage{amsfonts,amssymb}
\usepackage{multicol}
\def\normo#1{\left\|#1\right\|}

\def\abs#1{|#1|}
\def\aabs#1{\left|#1\right|}

\def\half#1{\frac{#1}{2}}
\def\norm#1{\|#1\|}

\newcommand{\N}{{\mathbb N}}

\newcommand{\R}{{\mathbb R}}
\newcommand{\C}{{\mathbb C}}
\newcommand{\Z}{{\mathbb Z}}

\newcommand{\les}{{\lesssim}}
\newcommand{\ges}{{\gtrsim}}

\theoremstyle{plain}
  \newtheorem{theorem}[subsection]{Theorem}
  \newtheorem{proposition}[subsection]{Proposition}
  \newtheorem{lemma}[subsection]{Lemma}
  \newtheorem{corollary}[subsection]{Corollary}
  \newtheorem{conjecture}[subsection]{Conjecture}

\theoremstyle{remark}
  \newtheorem{remark}[subsection]{Remark}

\theoremstyle{definition}
  \newtheorem{definition}[subsection]{Definition}

\begin{document}
\title{Improved Strichartz estimates for a class of dispersive equations in the radial case and their applications to nonlinear \\Schr\"odinger and wave equations}

\author{Zihua Guo$^1$, Yuzhao Wang$^2$
\\{\small $^1$LMAM, School of Mathematical Sciences, Peking University,
Beijing 100871, China}\\{\small \& School of Mathematics, Institute
for Advanced Study, NJ 08540, USA}\\{\small $^2$Department of
Mathematics and Physics, North China Electric Power
University}\\{\small Beijing 102206, China}\\{\small
E-mail:\,zihuaguo@math.pku.edu.cn,\, wangyuzhao2008@gmail.com}}

\date{}

\maketitle


\begin{abstract}
We prove some new Strichartz estimates for a class of dispersive
equations with radial initial data. In particular, we obtain up to
some endpoints the full radial Strichartz estimates for the
Schr\"odinger equation. The ideas of proof are based on Shao's ideas
\cite{Shao} and some ideas in \cite{GPW} to treat the
non-homogeneous case, while at the endpoint we need to use subtle
tools to overcome some logarithmic divergence. We also apply the
improved Strichartz estimates to the nonlinear problems. First, we
prove the small data scattering and large data LWP for the nonlinear
Schr\"odinger equation with radial critical $\dot{H}^s$ initial data
below $L^2$; Second, for radial data we improve the results of the
$\dot{H}^s\times \dot{H}^{s-1}$ well-posedness for the nonlinear
wave equation in \cite{SmithSogge}; Finally, we obtain the
well-posedness theory for the fractional order Schr\"odinger
equation in the radial case.

\end{abstract}


\section{Introduction}

In this paper, we study the Cauchy problems for a class of
dispersive equations which are of the following type:
\begin{align}\label{eq:abslinear}
i\partial_tu=-\phi(\sqrt{-\Delta})u+f,\quad u(0,x)=u_0(x),
\end{align}
where $\phi: \mathbb{R}^+\rightarrow \mathbb{R}$ is smooth away from
origin, $u(t,x):\R\times\R^n\rightarrow\C,\ n\geq 2$ is the unknown
function, $f(t,x)$ is the given function (e.g. $f=|u|^pu$ in the
nonlinear setting) and
$\phi(\sqrt{-\Delta})u=\mathscr{F}^{-1}\phi(|\xi|)\mathscr{F}u$.
Here $\mathscr{F}$ denotes the spatial Fourier transform, and
$\phi(|\xi|)$ is usually referred as the dispersion relation of
equation \eqref{eq:abslinear}. Many dispersive equations reduce to
this type, for instance, the Schr\"odinger equation ($\phi(r)=r^2$),
the wave equation ($\phi(r)=r$), the Klein-Gordon equation
($\phi(r)=\sqrt{1+r^2}$), the beam equation
($\phi(r)=\sqrt{1+r^4}$), and the fourth-order Schr\"odinger
equation ($\phi(r)=r^2+r^4$).

In the pioneered work \cite{STR}, Strichartz derived the priori
estimates of the solution to \eqref{eq:abslinear} in space-time norm
$L_t^qL_x^r$ by proving some Fourier restriction inequality. Later,
his results was improved via a dispersive estimate and duality
argument (cf. \cite{KT} and references therein). The dispersive
estimate
\begin{align}\label{eq:decayes}
\norm{e^{it\phi(\sqrt{-\Delta})}u_0}_{X} \lesssim
|t|^{-\theta}\norm{u_0}_{X'}
\end{align}
plays a crucial role, where $X'$ is the dual space of $X$. Applying
\eqref{eq:decayes}, together with a standard argument (cf.
\cite{KT}), we can immediately get the Strichartz estimates. For
instance, one can see from the explicit formula of the free
Schr\"odinger solution that
\[\norm{e^{it\Delta}u_0}_{L_x^\infty}\les
|t|^{-n/2}\norm{u_0}_{L_x^1}.\] In \cite{GPW}, the authors
systematically studied the dispersive estimates for
\eqref{eq:abslinear} by imposing some asymptotic conditions on
$\phi$.

As was explained in \cite{KT}, the full range of the non-retarded
Strichartz estimates for the Schr\"odinger equation were completely
known, while that of the retarded estimates remain open.
Surprisingly, if the initial data $u_0$ is radial, Shao \cite{Shao}
showed that the frequency localized non-retarded Strichartz
estimates for the Schr\"odinger equation allow a wider range. For
example, it was proved that
\begin{align}\label{eq:imschintro}
\norm{e^{it\Delta}P_k u_0}_{L^{q}_{t,x}(\R^{n+1})} \leq C
2^{(\frac{n}{2}-\frac{n+2}{q})k}\norm{u_0}_2
\end{align}
hold if $q>\frac{4n+2}{2n-1}$ and $u_0$ is radial. The proof relies
deeply on the radial assumption which eliminates the bad-type
evolution in the non-radial case (e.g. the Knapp counter-example).
Similar results hold for the wave equation, see \cite{Shao2}. It is
easy to see that equation \eqref{eq:abslinear} is
rotational-invariant, thus it is natural to ask whether one can get
better Strichartz estimates for the radial initial data than that
derived from the dispersive estimate.

The purposes of this paper are: first, to obtain the sharp range of
the type \eqref{eq:imschintro} for the improved Strichartz estimates
for equation \eqref{eq:abslinear} by using Shao's ideas \cite{Shao}
and the ideas in \cite{GPW}. Indeed, we will simplify some proofs
and overcome the difficulty caused by the lack of scaling invariance
by adapting some ideas in \cite{GPW}, moreover, we will prove that
\eqref{eq:imschintro} actually holds for $q=\frac{4n+2}{2n-1}$ by
dealing carefully with some logarithmic divergence; second, to apply
the improved Strichartz estimates to the nonlinear equations,
including nonlinear Schr\"odinger equation, nonlinear wave equation,
and nonlinear fractional-order Schr\"odinger equation. In order to
apply to the nonlinear problems, we will use the Christ-Kiselev
lemma to derive the retarded estimates from the non-retarded
estimates. For example, consider the nonlinear Schr\"odinger
equation
\[iu_t+\Delta u=\mu|u|^pu, \quad u(0,x)=u_0(x),\]
the well-posedness theory of which were deeply studied during the
past decades. We remark that the threshold of the regularity in
$\dot{H}^s$ for the strong well-posedness is $s\geq \max(0,s_c)$,
where $s_c$ is the scaling critical regularity, even in the case
that $L^2$ is subcritical in the sense of scaling. This can be seen
from the Galilean invariance (see \cite{BKPSV,CCT})
\[u(t,x)\rightarrow e^{-i|y|^2t+iy\cdot x}u(t,x-2ty), \quad y\in \R^d.\]
However, it is easy to see that the radial assumption breaks down
the Galilean invariance. Thus it is natural to expect that one may
go below $L^2$ in the radial case. This is indeed the case, which
will be discussed in details in Section 4.

In this paper, we consider the same class of $\phi$ as in
\cite{GPW}. In order to study the non-homogeneous case (e.g.
Klein-Gordon equation), we treat the high frequency and the low
frequency in different scales. As in \cite{GPW}, we will assume
$\phi: \mathbb{R}^+\rightarrow \mathbb{R}$ is smooth and satisfies
some of the following conditions:

(H1) There exists $m_1>0$, such that for any $\alpha \geq 2$ and
$\alpha \in \N$,
\begin{align*}
|\phi'(r)|\sim r^{m_1-1}\ and\ |\phi^{(\alpha)}(r)|\lesssim
r^{m_1-\alpha},\ \  \ r\geq 1.
\end{align*}

(H2)  There exists  $m_2>0$, such that for any $\alpha \geq 2$ and
$\alpha \in \N$,
\begin{align*}
|\phi'(r)|\sim r^{m_2-1}\ and\ |\phi^{(\alpha)}(r)|\lesssim
r^{m_2-\alpha},\ \ 0<r<1.
\end{align*}

(H3) There exists  $ \alpha_1$, such that
\begin{align*}
|\phi''(r)|\sim r^{\alpha_1-2}\ \ r\geq 1.
\end{align*}

(H4) there exists  $\alpha_2$, such that
\begin{align*}
|\phi''(r)|\sim r^{\alpha_2-2}\ \ 0<r<1.
\end{align*}

\begin{remark}\label{rem:malpha}
Heuristically, (H1) and (H3) reflect the dispersive effect in high
frequency. If $\phi$ satisfies (H1) and (H3), then $\alpha_1\leq
m_1$. Similarly, dispersive effect in low frequency is described by
(H2) and (H4). If $\phi$ satisfies (H2) and (H4), then $\alpha_2\geq
m_2$. The special case $\alpha_2= m_2$ happens in the most of time.
\end{remark}

For convenience, given $m_1,m_2,\alpha_1,\alpha_2\in \R$ as in
(H1)-(H4), we denote
\begin{align}\label{eq:malpha}
m(k)=
\begin{cases} m_1, &\text{for } k\ge 0,\\
 m_2, &\text{for } k< 0;
\end{cases}
\quad \text{and} \quad \alpha(k)=
 \begin{cases}  \alpha_1, &\text{for } k\ge 0,\\
  \alpha_2, &\text{for } k< 0.
 \end{cases}
\end{align}
Now we are ready to state our first result:

\begin{theorem}\label{thm:main}
Suppose $n\geq 2$, $k\in \Z$, $\phi: \mathbb{R}^+\rightarrow
\mathbb{R}$ is smooth away from origin, and $u_0$ is spherically
symmetric. If $\phi$ satisfies (H1) and (H2), then for
$\frac{2n}{n-1}<q\leq \infty$ we have
\begin{align}\label{eq:thm1}
\norm{S_\phi(t)P_k u_0}_{L^{q}_{t,x}(\R^{n+1})} \lesssim
2^{(\frac{n}{2}-\frac{n+m(k)}{q})k}\norm{u_0}_2,
\end{align}
Furthermore, if $\phi$ also satisfies (H3) and (H4), then for
$\frac{4n+2}{2n-1}\leq q\leq 6$ we have
\begin{align}\label{eq:thm2}
\norm{S_\phi(t)P_k u_0}_{L^{q}_{t,x}(\R^{n+1})} \les
2^{(\frac{n}{2}-\frac{n+m(k)}{q})k+
(\frac{1}{4}-\frac{1}{2q})(m(k)-\alpha(k))k}\norm{u_0}_2,
\end{align}
where $m(k),\alpha(k)$ are given by \eqref{eq:malpha}, and $P_k$ is
the Littlewood-Paley projector,
$S_\phi(t)=e^{it\phi(\sqrt{-\Delta})}$ is the dispersive group,
which will be defined later. Moreover, the range of $q$ is optimal
in the sense that \eqref{eq:thm1} fails if $q\leq \frac{2n}{n-1}$
and \eqref{eq:thm2} fails if $q<\frac{4n+2}{2n-1}$.
\end{theorem}

For the Schr\"odinger equation, $\phi(r)=r^2$ and satisfies
(H1)-(H4) with $m(k)=\alpha(k)=2$, then it follows immediately from
Theorem \ref{thm:main} that

\begin{corollary}\label{cor:sch}
Assume $n\geq 2$, $k\in \Z$, $\frac{4n+2}{2n-1}\leq q\leq \infty$.
Then there exists $C>0$ such that for $u_0\in L^2(\R^n)$ and $u_0$
is spherically symmetric, one has
\begin{align}\label{eq:cor}
\norm{e^{it\Delta}P_k u_0}_{L^{q}_{t,x}(\R^{n+1})} \leq C
2^{(\frac{n}{2}-\frac{n+2}{q})k}\norm{u_0}_2,
\end{align}
and the range of $q$ is optimal in the sense that \eqref{eq:cor}
fails if $q<\frac{4n+2}{2n-1}$.
\end{corollary}

\begin{remark}
Shao \cite{Shao} proved \eqref{eq:cor} for $q>\frac{4n+2}{2n-1}$.
For the wave equation, $\phi(r)=r$ and satisfies (H1)-(H2) with
$m(k)\equiv 1$, then \eqref{eq:thm1} reduces to the one given in
\cite{Shao2}. Interestingly, the range $q>\frac{2n}{n-1}$ is optimal
for the wave equation. It is worth noting that if
$q>\frac{2n}{n-1}$, \eqref{eq:thm1} gives better bound than
\eqref{eq:thm2} since $k[m(k)-\alpha(k)]\geq 0$ in view of Remark
\ref{rem:malpha}.
\end{remark}

We will apply Theorem \ref{thm:main} to some concrete equations.
Then using Christ-Kiselev lemma, we get the retarded Strichartz
estimates. In view of the classical Strichartz estimates, it is
natural to ask the sharp range of the mixed Strichartz estimates:
\[\norm{S_\phi(t)P_k u_0}_{L^{q}_tL^r_x(\R^{n+1})} \lesssim
C(k)\norm{u_0}_2.\] For this purpose, we restrict ourselves to the
simple case $\phi(r)=r^a$, $a>0$, namely, we consider the following
estimates
\begin{align}\label{eq:Stri}
\norm{e^{itD^a}P_kf}_{L_t^qL_x^r(\R\times \R^n)}\leq
C2^{k(\frac{n}{2}-\frac{a}{q}-\frac{n}{r})} \norm{f}_{L_x^2(\R^n)},
\end{align}
where $D=\sqrt{-\Delta}$, $a>0$. In this case, we have scaling
invariance, thus the proof is less complicated but still can be
adapted to the general case. We prove the following:
\begin{theorem}\label{thm:mainSch}
(a) Assume $a=1$ and $n\geq 3$. \eqref{eq:Stri} hold for all radial
functions $f\in L^2(\R^n)$ if and only if
\[(q,r)=(\infty,2)\quad or \quad 2\leq q \leq \infty, \frac{1}{q}+\frac{n-1}{r}<\frac{n-1}{2}.\]

(b) Assume $0<a\ne 1$ and $n\geq 2$. \eqref{eq:Stri} hold for all
radial functions $f\in L^2(\R^n)$ if
\[\frac{4n+2}{2n-1}\leq q\leq \infty, \frac{2}{q}+\frac{2n-1}{r}\leq n-\frac{1}{2} \quad or \quad 2\leq q<\frac{4n+2}{2n-1}, \frac{2}{q}+\frac{2n-1}{r}< n-\frac{1}{2}.\]
On the other hand, \eqref{eq:Stri} fail if $q>2$ or
$\frac{2}{q}+\frac{2n-1}{r}>n-\frac{1}{2}$.
\end{theorem}

\begin{remark}
The range of $(q,r)$ is indicated in Figure 1, where
$B=(\frac{n-3}{2n-2}, \frac{1}{2})$,
$C=(\frac{n-3}{2n-2},\frac{1}{2})$,
$D=(\frac{n-1}{2n},\frac{n-1}{2n})$, $B'=(\frac{n-2}{2n},
\frac{1}{2})$, $C'=(\frac{2n-3}{4n-2}, \frac{1}{2})$,
$D'=(\frac{2n-1}{4n+2},\frac{2n-1}{4n+2})$.
\begin{figure}
\begin{center}
\includegraphics{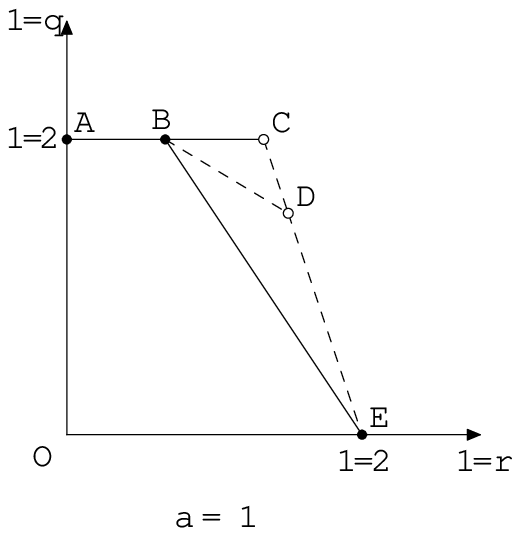}\quad\quad\quad \includegraphics{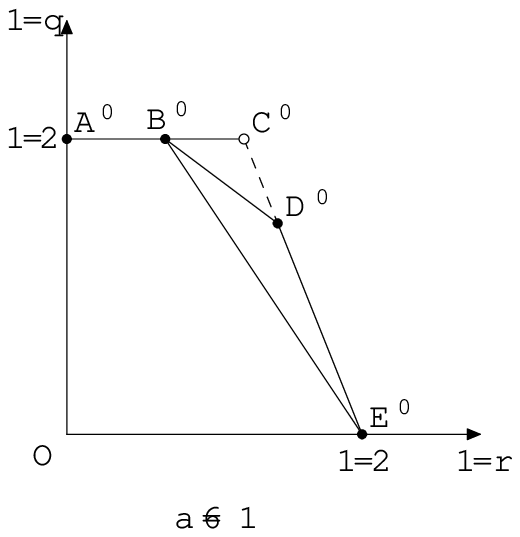}
\end{center}
\caption{Range of $(q,r)$ for \eqref{eq:Stri}}
\end{figure}
The results for the wave equation ($a=1$) are not new. The positive
results were due to \cite{KM, Sogge, Sterbenz, FangWang2}. The
counter-example was given in \cite{HiKu}.

On the other hand, for the Schr\"odinger equation, the results seem
to be new. We see that the picture is almost complete, except that
the segment $C'D'$ is unknown. In view of the positive results on
the segment $D'E'$, we conjecture that \eqref{eq:Stri} holds on the
segment $C'D'$, which is equivalent to the following
\begin{conjecture}\label{conj}
Assume $n\geq 2$ and $0<a\ne 1$. Then
\begin{align}\label{eq:conj}
\norm{e^{itD^a}P_0f}_{L_t^2L_x^{\frac{4n-2}{2n-3}}(\R\times
\R^n)}\leq C \norm{f}_{L_x^2(\R^n)},
\end{align}
holds for all radial function $f\in L_x^2(\R^n)$.
\end{conjecture}
This is very similar to the endpoint Strichart estimates in the
non-radial case that was studied in \cite{KT}. As is expected,
\eqref{eq:conj} is just ``logarithmically far" to be proved. Indeed,
we have for any $j\in \N$
\begin{align*}
\norm{e^{itD^a}P_0f}_{L_t^2L_x^{\frac{4n-2}{2n-3}}(\R\times
\{|x|\sim 2^j\})}\leq C \norm{f}_{L_x^2(\R^n)}.
\end{align*}
However, we can not adapt the method on $D'E'$ to overcome this
logarithmical divergence. See Remark \ref{rem:conj} below for more
discussions on \eqref{eq:conj}.
\end{remark}

Using these Strichartz estimates, we study the nonlinear problems
and prove some new results. For example, for the nonlinear
Schr\"odinger equation, we prove the following

\begin{theorem}\label{thm:iSch}
Assume $n\geq 2$, $0<p<4/n$, $s_{sch}=\frac{n}{2}-\frac{2}{p}$,
$\frac{1-n}{2n+1}\leq s_{sch}<0$, and $u_0$ is radial. If
$\norm{u_0}_{\dot{H}^{s_{sch}}}\leq \delta$ for some $\delta\ll 1$,
then there exists a unique global solution $u$ to
\[iu_t+\Delta u=\mu|u|^pu, \quad u(0,x)=u_0(x),\]
where $\mu=\pm 1$, such that $u\in C(\R:\dot{H}^{s_{sch}})\cap
L_{t,x}^{\frac{p(n+2)}{2}}(\R\times \R^n)$. Moreover, there exist
$u_{\pm}\in \dot{H}^{s_{sch}}$ such that
$\norm{u-S(t)u_{\pm}}_{\dot{H}^{s_{sch}}}\rightarrow 0$, as
$t\rightarrow \pm \infty$.
\end{theorem}

The index $\frac{1-n}{2n-1}$ is sharp for the critical GWP by our
methods. We actually obtain more results, see Theorem
\ref{thm:Schwp} below. For the nonlinear wave equation, we prove the
following
\begin{theorem}\label{thm:iWavewp}
Assume $n\geq 2$, $0<p<\frac{4}{n-1}$,
$s_{w}=\frac{n}{2}-\frac{2}{p}$, $\frac{1}{2n}<s_{w}<1/2$, and $u_0$
is radial. If
$\norm{u_0}_{\dot{H}^{s_{w}}}+\norm{u_1}_{\dot{H}^{s_w-1}}\leq
\delta$ for some $\delta\ll 1$, then there exists a unique global
solution $u$ to
\begin{align*}
&\partial_{tt}u-\Delta u=\mu|u|^pu,\quad (t,x)\in \R^{n+1},\\
&u(0)=u_0(x),\ u_t(0)=u_1(x),
\end{align*}
where $\mu=\pm 1$, such that $u\in C(\R:\dot{H}^{s_{w}})\cap
C^1(\R:\dot{H}^{s_w-1})\cap L_{t,x}^{\frac{2n+2}{n-2s_w}}(\R\times
\R^n)$, and there exist $(u_{\pm},v_\pm)\in \dot{H}^{s_{w}}\times
\dot{H}^{s_w-1}$ such that
$\norm{u-W'(t)u_{\pm}}_{\dot{H}^{s_w}}+\norm{u_t-W(t)v_{\pm}}_{\dot{H}^{s_w-1}}\rightarrow
0, \mbox{ as } t\rightarrow \pm \infty$.
\end{theorem}

Our results also hold for general nonlinearity, e.g. $F(u)$ with $F$
satisfying some conditions such as \eqref{eq:nonlinearF}. In
\cite{LindbladSogge}, Lindblad and Sogge studied the semilinear wave
equation with the same nonlinearity but with general non-radial
initial data. For example, for the nonlinearity $|u|^p$ they proved
small data scattering in $\dot{H}^s\times \dot{H}^{s-1}$ with
$s=\frac{n}{2}-\frac{2}{p-1}$ if $p\geq \frac{n+3}{n-1}$, and local
well-posedness if $s\geq s(p,n)$ for some $s(p,n)$. Thus we see that
their results covered the case $s_w\geq 1/2$ in Theorem
\ref{thm:iWavewp}, which is the main reason why we restrict
ourselves to the case $s_w<1/2$. In the same paper
\cite{LindbladSogge}, the authors actually showed that their results
are sharp by constructing some counter-examples. However, the
counter-examples for $s_w<1/2$ don't work for the radial case. Our
results in Theorem \ref{thm:iWavewp} showed that in the radial case
one can improve their results. Actually, we find a critical
regularity in the radial case $s_0(n)<\frac{1}{2n}$, which we will
discuss in details in Theorem \ref{thm:Wavewp}. In Section
\ref{sec:nonlinear}, we also study nonlinear fractional order
Schr\"odinger equation, and establish the well-posedness theory in
the radial case. We do not repeat the theorem here, but refer to
Theorem \ref{thm:FSchwp} below.

The fact that better well-posedness results hold in the radial case
was observed before, see \cite{Sogge, FangWang}, \cite{Hidano1,
Hidano2}. Our results generalize these results. In the non-radial
case, with additional angular regularity, one can also go below
$L^2$, see \cite{FangWang, JWY} and the references therein.
Actually, the results in \cite{FangWang} for the Schr\"odinger
equation are more generalized than ours but with different
resolution space. Our results for local well-posedness hold without
change for the inhomogeneous data $u_0\in H^s$ (see Remark
\ref{rem:inho}). It is then natural to ask whether \eqref{eq:cor}
and \eqref{eq:Stri} hold for non-radial functions with certain
angular regularity.

Throughout this paper, $C>1$ and $c<1$ will denote positive
universal constants, which can be different at different places.
$A\lesssim B$ means that $A\leq CB$,  and $A \sim B$ stands for
$A\lesssim B$ and $B\lesssim A$. We use $\hat{f}(\xi)$ and
$\mathscr{F}(f)$ to denote the spatial Fourier transform of $f$ on
$\R^{n}$ defined by
\[\hat{f}(\xi)=\int_{\R^{n}}f(x)e^{-ix\cdot\xi}dx.\]
We denote by $p'$ the dual number of $p \in [1,\infty]$, i.e.,
$1/p+1/p'=1$. Let $\Phi(x): \mathbb{R}\to [0,1]$ be a non-negative,
smooth even function such that \mbox{supp}$\Phi \subseteq
\{x:\abs{x}\leq 2\}$, and $\Phi(x)=1$, if $\abs{x}\leq 1$. Let
$\psi(x)=\Phi(x)-\Phi(2x)$, and $P_k$ be the Littlewood-Paley
projector for $k\in \Z$, namely
\[P_k
f=\mathscr{F}^{-1}\psi(2^{-k} |\xi|)\mathscr{F}f,\quad P_{\leq 0}
f=\mathscr{F}^{-1}\Phi(|\xi|)\mathscr{F} f.\] We denote by
$S_\phi(t)$ the evolution group related to \eqref{eq:abslinear},
defined as
\[
S_\phi(t)u_0(x) = e^{it\phi(\sqrt{-\Delta})} u_0(x) = c_n
\int_{\R^n} e^{ix\cdot \xi} e^{it\phi(|\xi|)} \hat u_0 (\xi)\, d\xi.
\]
We will use Lebesgue spaces $L^p:=L^p(\mathbb{R}^n)$, $\|\cdot\|_p
:=\|\cdot\|_{L^p}$. and the space-time norm $L^q_tL^r_x$ of $f$ on
$\R\times \Omega$ by
$$
\|f(t,x)\|_{L^q_tL^r_x(\R\times \Omega)}=\Big\| \| f(t,x
)\|_{L^r_x(\Omega)} \Big\|_{L^q_t(\R)},
$$
where $\Omega \subset \R^n$. When $q=r$, we abbreviate it by
$L^{q}_{t, x}(\R\times \Omega)$.

The rest of this paper is organized as follows. In Section 2, we
prove Theorem \ref{thm:main}. In Section 3 we present the
applications of Theorem \ref{thm:main} to some concrete equations.
In Section 4, we apply the improved Strichartz estimates to the
nonlinear problems.

\section{Proof of Theorem \ref{thm:main} and Theroem
\ref{thm:mainSch}}

First we prove Theorem \ref{thm:main}. We will adapt some ideas in
\cite{GPW} and Shao's ideas \cite{Shao}. However, there is a new
difficulty for the endpoint case $q=\frac{4n+2}{2n-1}$ in
\eqref{eq:thm2} due to some logarithmic divergence. Fortunately
enough, this logarithmic divergence can be overcome by using a
subtle tool: double weight Hardy-Littlewood-Sobolev inequality. On
the other hand, the logarithmic divergence for the endpoint
$q=\frac{2n}{n-1}$ in \eqref{eq:thm1} is essential. We present the
proof by the following three steps:

\noindent{\bf Step 1.} Non-endpoint: $q>\frac{2n}{n-1}$ in
\eqref{eq:thm1}, $q>\frac{4n+2}{2n-1}$ in \eqref{eq:thm2}.

For $j\in \Z$, denote
\[A_j:=\{x\in \R^{n}: 2^{j-1}\le |x|< 2^{j}\}, \quad I_j=[2^{j-1},2^{j}).\]
Fixing $k\in \Z$,
we decompose $\|S_\phi(t)\Delta_k
u_0(x)\|_{L^q_{t,x}(\R\times\R^n)}$ and get
\begin{align}\label{deco}
\|S_\phi(t)P_k u_0\|_{L^q_{t,x}(\R^{n+1})}
\le&\sum_{j\in\Z}\|S_\phi(t)P_k u_0\|_{L^q_{t,x}(\R\times A_j)}\nonumber\\
=&\sum_{j+k\le 1}\|S_\phi(t)P_k u_0\|_{L^q_{t,x}(\R\times A_j)}
+\sum_{j+k\ge 2}\|S_\phi(t)P_k u_0\|_{L^q_{t,x}(\R\times A_j)}.
\end{align}
The main tasks reduce to estimate $\|S_\phi(t)P_k
u_0\|_{L^q_{t,x}(\R\times A_j)}$. It is easy to see that
$S_\phi(t)P_k u_0$ is spherically symmetric in space if $u_0$ is
radial. Thus we can rewrite it in an integral form related to the
Bessel function. The two parts $j+k\leq 1$ and $j+k\geq 2$ exploit
different properties of the Bessel function. We give the estimates
of the two parts in the following two propositions.

\begin{proposition}\label{p1}
Assume $u_0\in L^2(\R^n)$, $u_0$ is radial, and $\phi$ satisfies
(H1) and (H2). Then if $k,j \in \Z$ with $j+k\le 1$ and $2\leq q\leq
\infty$, we have
\begin{align}\label{kj<0}
\|S_{\phi}(t)P_k u_0(x)\|_{L^q_{t,x}(\R\times A_j)}\lesssim
2^{\frac{nj}{q}}2^{(\frac{n}{2}-\frac{m(k)}{q})k}\|P_k u_0\|_{L^2},
\end{align}
where $m(k)$ is given by \eqref{eq:malpha}.
\end{proposition}

\begin{proposition}\label{p2}
Assume $u_0\in L^2(\R^n)$, $u_0$ is radial, and $\phi$ satisfies
(H1) and (H2). Then if $k,j \in \Z$ with $j+k\ge 2$ and $2\leq q\leq
\infty$, we have
\begin{align}\label{kj>0}
\|S_{\phi}(t)P_k u_0(x)\|_{L^q_{t,x}(\R\times A_j)}\lesssim
2^{(\frac{n}{q}-\frac{n-1}{2})j}2^{(\frac{1}{2}
-\frac{m(k)}{q})k}\|P_k u_0\|_{L^2}.
\end{align}
Furthermore, if $\phi$ also satisfies (H3) and (H4), then for $2\leq
q\leq 6$
\begin{align}\label{kj>0}
\|S_{\phi}(t)P_k u_0(x)\|_{L^q_{t,x}(\R\times A_j)}\lesssim
2^{(\frac{2n+1}{2q}-\frac{2n-1}{4})j}2^{(\frac{-3m(k)+\alpha(k)+1}{2q}+\frac{m(k)-\alpha(k)+1}{4})k}\|P_k
u_0\|_{L^2},
\end{align}
where $m(k),\alpha(k)$ is given by \eqref{eq:malpha}.
\end{proposition}

We postpone the proofs of Proposition \ref{p1} and Proposition
\ref{p2}, and first use them to complete the proof of Theorem
\ref{thm:main} in the non-endpoint case.

\begin{proof}[\bf Proof of Theorem \ref{thm:main} (non endpoint)]
We may assume $q<\infty$. Assume first that $\phi$ satisfies (H1)
and (H2). From \eqref{deco}, Proposition \ref{p1} and Proposition
\ref{p2}, we get
\begin{align*}
\|S_\phi(t)P_k u_0(x)\|_{L^q_{t,x}(\R^{n+1})}\les &\sum_{j+k\le
1}2^{\frac{nj}{q}}2^{(\frac{n}{2}-\frac{m(k)}{q})k}\|P_k
u_0\|_{L^2}\\
&+\sum_{j+k\ge 2}2^{(\frac{n}{q}-\frac{n-1}{2})j}2^{(\frac{1}{2}
-\frac{m(k)}{q})k}\|P_k u_0\|_{L^2}\\
\les&2^{(\frac{n}{2}-\frac{m(k)}{q}-\frac{n}{q})k}\|P_k u_0\|_{L^2},
\end{align*}
since $q>\frac{2n}{n-1}$ then $\frac{n}{q}-\frac{n-1}{2}<0$. Thus
\eqref{eq:thm1} is proved. Now we assume $\phi$ also satisfies (H3)
and (H4), then
\begin{align*}
\|S_\phi(t)P_k u_0(x)\|_{L^q_{t,x}(\R^{n+1})}\les &\sum_{j+k\le 1}2^{\frac{nj}{q}}2^{(\frac{n}{2}-\frac{m(k)}{q})k}\|P_k u_0\|_{L^2}\\
&+\sum_{j+k\ge
2}2^{(\frac{2n+1}{2q}-\frac{2n-1}{4})j}2^{(\frac{-3m(k)+\alpha(k)+1}{2q}-\frac{m(k)-\alpha(k)+1}{4})k}\|P_k
u_0\|_{L^2}.
\end{align*}
Note that if $q>\frac{4n+2}{2n-1}$, then
$\frac{2n+1}{2q}-\frac{2n-1}{4}<0$. Thus we can sum over $j$ and
bound the quantity above by
\[
C\Big[2^{(\frac{n}{2}-\frac{n+m(k)}{q})k+ (\frac{1}{4}-\frac{1}{2q})[m(k)-\alpha(k)]k}+
 2^{(\frac{n}{2}-\frac{n+m(k)}{q}))k}\Big]\|P_k u_0\|_{L^2}.
\]
Which is sufficient for \eqref{eq:thm2} since
$(\frac{1}{4}-\frac{1}{2q})[m(k)-\alpha(k)]k \ge 0$ in view of
Remark \ref{rem:malpha}.
\end{proof}

It remains to prove Proposition \ref{p1} and Proposition \ref{p2}.
The proof relies heavily on the radial properties. In particular, we
will use the Fourier-Bessel formula. We denote by $J_m(r)$ the
Bessel function:
\begin{align*}
J_m(r)=\frac{(r/2)^m}{\Gamma(m+1/2)\pi^{1/2}}
\int_{-1}^1e^{irt}(1-t^2)^{m-1/2}dt, \ \ m>-1/2.
\end{align*}
We first list some properties of $J_m(r)$ that will be used in the
following lemma. For their proof we refer the readers to
\cite{Stein1993}.
\begin{lemma}[Properties of the Bessel function]\label{lem:Bessel}
We have for $0<r<\infty$ and $m>-\half 1$

(i) $J_m(r)\leq Cr^m$,

(ii) $J_m(r)\leq Cr^{-\half 1}$.
\end{lemma}
It is well known that if $f(x)=g(|x|)$ is radial, then the Fourier
transform of $f$ is also radial (cf. \cite{Stein1971}), and
\begin{align}\label{e8}
\hat{f} (\xi)= 2\pi  \int_{0}^{\infty} g(s) s^{n-1}(s
|\xi|)^{-\half{n-2}}J_{\frac{n-2}{2}}(s|\xi|)ds.
\end{align}
Thus if $\widehat{u_0}(\xi)=h(|\xi|)$ is radial, then $S_\phi(t)P_k
u_0=F(t,|x|)$, and
\begin{align}\label{eq:radialfree}
F(t,|x|)= 2\pi  \int_{0}^{\infty}
e^{it\phi(s)}\psi_k(s)h(s)s^{n-1}(s
|x|)^{-\half{n-2}}J_{\frac{n-2}{2}}(s|x|)ds.
\end{align}
The issues reduce to a one-dimensional problem involving Bessel
function. We will use the following local smoothing effect type
estimate.

\begin{lemma}\label{lemma:smoothing effect}
Suppose $k\in \Z$, $\varphi \in L^2(\R)$ and $\phi$ satisfies (H1)
and (H2). Then for $2\leq q\leq \infty$, we have
\[
\Big\|\int_{\R} \psi_k(s) \varphi(s) e^{-it\phi(s)} \, ds\Big\|_{L^q_t}\les 2^{(\frac{1}{2}-\frac{m(k)}{q})k}\|\psi_k \varphi\|_{L^2}
\]
where $m(k)$ is defined in \eqref{eq:malpha}.
\end{lemma}

\begin{proof}
It is easy to see that in the support of $\psi_k$, $\phi$ is
invertible and we denote $\phi^{-1}$ to be the inverse of $\phi$. By
the change of variable $a=\phi(s)$, we get
\[
\Big\|\int_{\R} \psi_k(s)\varphi(s) e^{-it\phi(s)} \,
ds\Big\|_{L^q_t} = \Big\|\int_{\R} \psi_k(\phi^{-1}(a)) e^{-ita}
\frac{\varphi(\phi^{-1}(a))
}{|\phi'(\phi^{-1}(a))|}da\Big\|_{L^q_t}.
\]
Then from the Hausdorff-Young inequality and change of variable
$s=\phi(a)$, we get the quantity above is bounded by
\[
C\Big\|\psi_k(\phi^{-1}(a))\frac{\varphi(\phi^{-1}(a))}{|\phi'(\phi^{-1}(a))|}\Big\|_{L^{q'}_a}
=C\Big\|\psi_k(s)\frac{\varphi(s)}{|\phi'(s)|^{\frac{1}{q}}}\Big\|_{L^{q'}_s},
\]
From the condition we have $\phi'(s) \sim 2^{(m(k)-1)k}$ in the
support of $\psi_k$, and then by H\"older inequality we can bound
the quantity above by
\[
C2^{\frac{-m(k)+1}{q}k}2^{(\frac{1}{q'}-\frac{1}{2})k}\|\psi_k(s) \varphi(s)\|_{L^2_s}=C2^{(\frac{1}{2}-\frac{m(k)}{q})k}\|\psi_k\varphi\|_{L^2}
\]
Thus we finish the proof.
\end{proof}

\begin{lemma}[Strichartz estimate]\label{lemma:str}
Suppose $\varphi\in L^2(\R)$ and $\phi$ satisfies one of H(3) and
H(4). Then for $k\in \Z$, we have
\[
\Big\|\int_{\R} \psi_k(s) \varphi(s) e^{irs-it\phi(s)} \,
ds\Big\|_{L^6_tL_r^6}\les
2^{(\frac{1}{3}-\frac{\alpha(k)}{6})k}\|\psi_k\varphi \|_{L^2},
\]
where $\alpha(k)$ is defined in \eqref{eq:malpha}.
\end{lemma}

\begin{proof}
Since $\phi$ satisfies (H3) and (H4), then by Theorem 1 in
\cite{GPW}, we have the decay estimate
\[
\Big\|\int_{\R} \psi_k(s) \varphi(s) e^{irs-it\phi(s)} \,
ds\Big\|_{L^\infty_r}\les
|t|^{-\frac{1}{2}}2^{(1-\frac{\alpha(k)}{2})}
\|\mathscr{F}^{-1}[\psi_k \varphi]\|_{L^1}.
\]
Then Lemma 2 follows immediately from Proposition 1 in \cite{GPW},
also see \cite{KT}.
\end{proof}

\begin{proof}[\bf Proof of Proposition \ref{p1}]
We get from \eqref{eq:radialfree} and Lemma \ref{lem:Bessel} (i) and
Lemma \ref{lemma:smoothing effect} that
\begin{align*}
\|S_\phi(t)P_k u_0(x)\|_{L^q_{t,x}(\R\times
A_j)}\les&\norm{F_k(t,r)r^{\frac{n-1}{q}}}_{L_t^qL_{I_j}^q}\\
\les&2^{(\frac{1}{2}-\frac{m(k)}{q})k}\norm{\psi_k(s)h(s)s^{n-1}r^{\frac{n-1}{q}}}_{L_{r\in
I_j}^qL_s^2}\\
\les&2^{\frac{nj}{q}}2^{(\frac{n}{2}-\frac{m(k)}{q})k}\norm{\psi_k(s)h(s)s^{\frac{n-1}{2}}}_{L_s^2}
\end{align*}
which completes the proof of Proposition \ref{p1}, since
$\norm{\psi_k(s)h(s)s^{\frac{n-1}{2}}}_{L_s^2}=\norm{P_k
u_0}_{L^2}$.
\end{proof}

It remains to prove Proposition \ref{p2}. We will use the decay
properties at the infinity of the Bessel function. More precisely,
\begin{align}\label{eq:Besselinfty}
J_{\frac{n-2}{2}}(r)=\frac{e^{i(r-(n-1)\pi/4)}+e^{-i(r-(n-1)\pi/4)}}{2r^{1/2}}+d_nr^{\frac{n-2}{2}}e^{-ir}E_+(r)-e_nr^{\frac{n-2}{2}}e^{ir}E_-(r),
\end{align}
where $E_\pm(r)\les r^{-(n+1)/2}$ if $r\ge1$, $d_n,e_n$ are
constants, see \cite{Stein1993}. Inserting \eqref{eq:Besselinfty}
into \eqref{eq:radialfree}, we then divide $F(t,|x|)$ into two
parts: the main term and the error term, namely
\begin{align}\label{eq:decfree}
F(t,|x|)=M(t,|x|)+E(t,|x|)
\end{align}
with
\begin{align*}
{M}(t,r)=&c_nr^{-\frac{n-1}{2}}\int_{\R}
\psi_k(s)h(s)s^{\frac{n-1}{2}} e^{i(rs-t\phi(s))}
ds+\bar{c_n}r^{-\frac{n-1}{2}}\int_{\R} \psi_k(s)
h(s)s^{\frac{n-1}{2}}e^{-i(rs+t\phi(s))} ds,\\
{E}(t,r)=&c_1\int_{\R} \psi_k(s)h(s)s^{n-1}e^{-it\phi(s)-
irs}E_{+}(rs)ds-c_2\int_{\R} \psi_k(s)h(s)s^{n-1}e^{-it\phi(s)+
irs}E_{-}(rs)ds.
\end{align*}

First we estimate the error term ${E(t,|x|)}$ in the following
Lemma.
\begin{lemma}\label{lem:error}
Assume $\phi$ satisfies (H1) and (H2). If $j+k \ge 2$ and $2\leq
q\leq \infty$, we have
\begin{align}\label{eq:lemerror}
\|E(t,|x|)\|_{L^q_{t,x}(\R\times A_j)}\lesssim 2^{(-\frac
{n+1}{2}+\frac{n}{q})j} 2^{-(\frac {1}{2}+\frac{m(k)}{q})k}\|P_k
u_0\|_{L^{2}}.
\end{align}
\end{lemma}
\begin{proof}
As in the proof of Proposition \ref{p1}, we have
\begin{align*}
\|E(t,|x|)\|_{L^q_{t,x}(\R\times
A_j)}\les&\norm{E(t,r)r^{\frac{n-1}{q}}}_{L_t^qL_{I_j}^q}\\
\les&2^{(\frac
{1}{2}-\frac{m(k)}{q})k}\norm{\psi_k(s)F(s)s^{n-1}r^{\frac{n-1}{q}}E_\pm(rs)}_{L_{r\in
I_j}^qL_s^2}\\
\les&2^{-(\frac
{1}{2}+\frac{m(k)}{q})k}2^{j(\frac{n}{q}-\frac{n+1}{2})}\norm{\psi_k(s)F(s)s^{\frac{n-1}{2}}}_{L_s^2},
\end{align*}
where we used the fact $|E_\pm(r)|\les r^{-(n+1)/2}$. Thus we
complete the proof.
\end{proof}

Next we estimate the main term $M(t,|x|)$ in the following Lemma.

\begin{lemma}\label{lem:main}
(a) Assume $\phi$ satisfies (H1) and (H2). If $j+k \ge 2$, we have
\begin{align}
\|M(t,|x|)\|_{L^2_{t,x}(\R\times A_j)}\lesssim&
2^{j/2}2^{\frac{1-m(k)}{2}k}\|P_ku_0\|_{L^2},\label{eq:mainterm2}\\
\|M(t,|x|)\|_{L^\infty_{t,x}(\R\times
A_j)}\lesssim&2^{-j(n-1)/2}2^{k/2}\|P_ku_0\|_{L^2}.\label{eq:mainterminfty}
\end{align}

(b) Assume $\phi$ satisfies (H3) and (H4). If $j+k \ge 2$, we have
\begin{align}\label{eq:mainterm6}
\|M(t,x)\|_{L^6_{t,x}(\R\times A_j)}\lesssim
2^{-\frac{n-1}{3}j}2^{(\frac{1}{3}-\frac{\alpha(k)}{6})k}\|P_ku_0\|_{L^2}.
\end{align}
\end{lemma}
\begin{proof}
From symmetry it suffices to estimate the first term in $M(t,|x|)$.
We get from Lemma \ref{lemma:smoothing effect} with $q=2$ that
\begin{align*}
\norm{M(t,|x|)}_{L^2_{t,x}(\R\times
A_j)}\les&\norm{M(t,r)r^{\frac{n-1}{2}}}_{L_t^2L_{I_j}^2}\\
\les&\normo{\int_\R
\psi_k(s)h(s)s^{\frac{n-1}{2}}e^{i(rs-t\phi(s))}ds}_{L_{I_j}^2L_t^2}\\
\les&2^{j/2}2^{-\frac{m(k)-1}{2}k}\norm{\psi_k(s)h(s)s^{\frac{n-1}{2}}}_{L_s^2},
\end{align*}
which gives the first inequality, as desired. Similarly,
\begin{align*}
\norm{M(t,|x|)}_{L^\infty_{t,x}(\R\times
A_j)}\les&\norm{M(t,r)}_{L_t^\infty L_{I_j}^\infty}\\
\les&2^{-j(n-1)/2}\normo{\int_\R
\psi_k(s)h(s)s^{\frac{n-1}{2}}e^{i(rs-t\phi(s))}ds}_{L_{I_j}^\infty L_t^\infty}\\
\les&2^{-j(n-1)/2}2^{k/2}\norm{\psi_k(s)h(s)s^{\frac{n-1}{2}}}_{L_s^2},
\end{align*}

To prove (b), we get from Lemma \ref{lemma:str} that
\begin{align*}
\|M(t,|x|)\|_{L^6_{t,x}(\R\times A_j)}\les&\norm{M(t,r)r^{\frac{n-1}{6}}}_{L_t^6L_{I_j}^6}\\
\les& 2^{-(n-1)j/3}\normo{\int_\R
\psi_k(s)h(s)s^{\frac{n-1}{2}}e^{i(rs-t\phi(s))}ds}_{L_t^6L_{I_j}^6}\\
\les&2^{-(n-1)j/3}2^{(\frac{1}{3}-\frac{\alpha(k)}{6})k}\|\psi_k(s)h(s)s^{\frac{n-1}{2}}\|_{L^2_s},
\end{align*}
which completes the proof of the lemma.
\end{proof}

Now we are ready to prove Proposition \ref{p2}.

\begin{proof}[\bf Proof of Proposition \ref{p2}]
If $\phi$ satisfies (H1) and (H2), then by interpolating
\eqref{eq:mainterm2} and \eqref{eq:mainterminfty} we get that for
$2\leq q\leq \infty$
\begin{align}\label{mainq2}
\|M(t,x)\|_{L^q_{t,x}(\R\times A_j)}\les
2^{(\frac{n}{q}-\frac{n-1}{2})j}2^{(\frac{1}{2}
-\frac{m(k)}{q})k}\|P_k u_0\|_{L^2}.
\end{align}
Then from Lemma \ref{lem:error} we get for $2\leq q\leq \infty$
\begin{align*}
\|S_\phi(t)P_k u_0(x)\|_{L^q_{t,x}(\R\times A_j)}\les&
\|E(t,x)\|_{L^q_{t,x}(\R\times A_j)}+\|M(t,x)\|_{L^q_{t,x}(\R\times A_j)}\\
\les & 2^{(\frac{n}{q}-\frac{n-1}{2})j}2^{(\frac{1}{2}
-\frac{m(k)}{q})k}\|P_k u_0\|_{L^2}.
\end{align*}
Moreover, if $\phi$ also satisfies (H3) and (H4), then by
interpolating \eqref{eq:mainterm2} and \eqref{eq:mainterm6} we get
that for $2\leq q\leq 6$
\begin{align}\label{mainq2}
\|M(t,x)\|_{L^q_{t,x}(\R\times A_j)}&\les
2^{(\frac{2n+1}{2q}-\frac{2n-1}{4})j}2^{(\frac{-3m(k)+\alpha(k)+1}{2q}
+\frac{m(k)-\alpha(k)+1}{4})k}\|P_k u_0\|_{L^2}.
\end{align}
Thus, in view of Lemma \ref{lem:error} and \eqref{mainq2}, the
left-hand side of \eqref{kj>0} can be bounded by
\begin{align*}
\|S_\phi(t)P_k u_0(x)\|_{L^q_{t,x}(\R\times A_j)}\les&
\|{E}(t,x)\|_{L^q_{t,x}(\R\times A_j)}+\|M(t,x)\|_{L^q_{t,x}(\R\times A_j)}\\
\les & (C_1(k,j)+C_2(k,j))\|P_k u_0\|_{L^2}
\end{align*}
where
\begin{align*}
C_1(k,j)=&2^{(-\frac {n+1}{2}+\frac{n}{q})j}2^{-(\frac {1}{2}+\frac{m(k)}{q})k},\\
C_2(k,j)=&2^{(\frac{2n+1}{2q}-\frac{2n-1}{4})j}2^{(\frac{-3m(k)+\alpha(k)+1}{2q}
+\frac{m(k)-\alpha(k)+1}{4})k}.
\end{align*}
It remains to prove $C_1(k,j)\le C_2(k,j)$. Actually, by simple
calculation we get
\begin{align*}
\frac{C_2(k,j)}{C_1(k,j)} = & 2^{(\frac{2n+1}{2q}-\frac{2n-1}{4}+\frac {n+1}{2}-\frac{n}{q})j+(\frac{-3m(k)+\alpha(k)+1}{2q}
+\frac{m(k)-\alpha(k)+1}{4}+\frac {1}{2}+\frac{m(k)}{q})k}\\
=&2^{(j+k)(\frac{1}{2q}+\frac{3}{4})+(\frac{1}{4}-\frac{1}{2q})(m(k)-\alpha(k))k}.
\end{align*}
It is easy to see that
\[
(j+k)\Big(\frac{1}{2q}+\frac{3}{4}\Big)+\Big(\frac{1}{4}-\frac{1}{2q}\Big)(m(k)-\alpha(k))k\ge
1,
\]
since $j+k\ge 2$ and $(m(k)-\alpha(k))k\ge 0$ in view of Remark
\ref{rem:malpha}. Thus we finish the proof.
\end{proof}

\noindent{\bf Step 2.} Endpoint: $q=\frac{4n+2}{2n-1}$ in
\eqref{eq:thm2}.

From step 1 we see that in this case we just fail to sum over $j\geq
2-k$. To overcome this, we do not decompose for large $j$. The main
tools are the Van der Corput Lemma \cite{Stein1993} and double
weight Hardy-Littlewood-Sobolev inequalities \cite{Stein1958}:
\begin{lemma}[Van der Corput]\label{lem:vander}
Assume $\psi \in C_0^\infty(\R)$ and $P\in C^2(\R)$ is a real-valued
function such that $|P''(\xi)|\geq \lambda$ in the support of
$\psi$. Then
\[\aabs{\int e^{iP(\xi)}\psi(\xi)d\xi}\leq C\lambda^{-1/2}(\norm{\psi}_\infty+\norm{\psi'}_1).\]
\end{lemma}

\begin{lemma}\label{lem:DWHLS}
If $1<r,s<\infty$, $1/r+1/s\geq 1$, $0<\lambda<d$, $\alpha+\beta\geq
0$ and
\[1-\frac{1}{r}-\frac{\lambda}{d}<\frac{\alpha}{d}<1-\frac{1}{r},\ \frac{1}{r}+\frac{1}{s}+\frac{\lambda+\alpha+\beta}{d}=2,\]
then
\[\aabs{\int_{\R^d}\int_{\R^d}\frac{f(x)g(y)}{|x|^\alpha |x-y|^\lambda|y|^\beta}dxdy}\leq C_{\alpha,\beta,s,\lambda,d}\norm{f}_r \norm{g}_s.\]
\end{lemma}

Now we proceed to prove \eqref{eq:thm2} for $q=\frac{4n+2}{2n-1}$.
Obviously, we have
\begin{align*}
\|S_\phi(t)P_k u_0\|_{L^q_{t,x}(\R^{n+1})} \le&\sum_{j\le
1-k}\|S_\phi(t)P_k u_0\|_{L^q_{t,x}(\R\times A_j)}
+\|S_\phi(t)P_k u_0\|_{L^q_{t,x}(\R\times \{|x|\geq 2^{1-k}\})}\\
:=&I+II.
\end{align*}
From step 1 we see that the term $I$ is bounded as desired. It
remains to bound the term $II$. Using \eqref{eq:decfree} we get
\begin{align*}
II \le&\|M(t,|x|)\|_{L^q_{t,x}(\R\times \{|x|\geq 2^{1-k}\})}+\|E(t,|x|)\|_{L^q_{t,x}(\R\times \{|x|\geq 2^{1-k}\})}\\
:=&II_1+II_2.
\end{align*}
From step 1 we see that the term $II_2$ is bounded as desired. Thus,
it remains to bound the term $II_1$. From symmetry, it suffices to
prove
\begin{align*}
&\normo{1_{[2^{1-k},\infty)}(r)r^{(\frac{1}{q}-\frac{1}{2})(n-1)}\int_{\R}
\psi_k(s)h(s)s^{\frac{n-1}{2}} e^{i(rs-t\phi(s))}
ds}_{L^q_{t,r}}\\
&\les 2^{(\frac{n}{2}-\frac{n+m(k)}{q})k+
(\frac{1}{4}-\frac{1}{2q})(m(k)-\alpha(k))k}\norm{h(s)s^{\frac{n-1}{2}}}_2
\end{align*}
which follows from the following estimate
\begin{align}\label{eq:weistr}
\normo{|r|^{(\frac{1}{q}-\frac{1}{2})(n-1)}\int_{\R}
\psi_0(s)h(s)e^{i(rs-t2^{-km(k)}\phi(2^ks))} ds}_{L^q_{t,r}}\les
2^{(\frac{1}{4}-\frac{1}{2q})(m(k)-\alpha(k))k} \norm{h}_2.
\end{align}

It remains to prove \eqref{eq:weistr}. Since $\psi_0(s)$ is
supported in $\{s\sim 1\}$, then from (H1)-(H4) we get that
$\phi_k=2^{-km(k)}\phi(2^ks)$ has an inverse denoted by
$\eta_k=\phi_k^{-1}: range(\phi_k)\to \{s\sim 1\}$, moreover,
\begin{align}\label{eq:etaderi}
|\eta_k'|\sim 1, \quad |\eta''|\sim 2^{k(\alpha(k)-m(k))}.
\end{align}
By a change of variable $s=\eta_k(\mu)$, we get that
\eqref{eq:weistr} is equivalent to
\begin{align}\label{eq:weistr2}
\normo{|r|^{(\frac{1}{q}-\frac{1}{2})(n-1)}\int_{\R}
\psi_0(\eta_k(\mu))h(\mu)e^{i(r\eta_k(\mu)-t\mu)}
d\mu}_{L^q_{t,r}}\les
2^{(\frac{1}{4}-\frac{1}{2q})(m(k)-\alpha(k))k} \norm{h}_2
\end{align}
For $f\in L^2(\R)$, define operator
\[Tf(x,t)=|x|^{(\frac{1}{q}-\frac{1}{2})(n-1)}\int_{\R}
\psi_0(\eta_k(\mu))f(\mu)e^{i(x\eta_k(\mu)-t\mu)} d\mu.\] It
suffices to prove $\norm{T}_{L^2\to L^q_{t,x}}\les
2^{(\frac{1}{4}-\frac{1}{2q})(m(k)-\alpha(k))k}$. By duality, we
have
\[T^*g(\mu)=\psi_0(\eta_k(\mu))\int_{\R\times \R} e^{-i(x\eta_k(\mu)-t\mu)}|x|^{(\frac{1}{q}-\frac{1}{2})(n-1)}g(x,t)dxdt.\]
By the $TT^*$ arguments, it suffices to prove
\[\norm{TT^*g}_{L^q}\les 2^{(\frac{1}{2}-\frac{1}{q})(m(k)-\alpha(k))k} \norm{g}_{L^{q'}}.\]
From the definition we have
\begin{align*}
TT^*g(x,t)=&|x|^{(\frac{1}{q}-\frac{1}{2})(n-1)}\int
\psi_0^2(\eta_k(\mu))
e^{-i(y\eta_k(\mu)-\tau \mu)}{|y|}^{(\frac{1}{q}-\frac{1}{2})(n-1)}g(y,\tau) e^{i(x\eta_k(\mu)-t\mu)}d\mu dy d\tau\\
=&|x|^{(\frac{1}{q}-\frac{1}{2})(n-1)}\int_{\R^2}
K(x-y,t-\tau){|y|}^{(\frac{1}{q}-\frac{1}{2})(n-1)}g(y,\tau)dy
d\tau,
\end{align*}
where
\[K(x-y,t-\tau)=\int \psi_0^2(\eta_k(\mu))e^{i[(x-y)\eta_k(\mu)-(t-\tau)\mu]}d\mu.\]
Using Plancherel's equality, we get
\[\normo{\int K(x-y,t-\tau)g(y,\tau)d\tau}_{L^2_t}\les \norm{g(y,\cdot)}_{L^{2}}.\]
On the other hand, it follows from Van der Corput lemma and
\eqref{eq:etaderi} that
\[|K(x-y,t-\tau)|\les 2^{\frac{k(m(k)-\alpha(k))}{2}}|x-y|^{-1/2}.\]
Then by interpolation we have
\[\normo{\int K(x-y,t-\tau)g(y,\tau)d\tau}_{L^q_t}\les
2^{k(m(k)-\alpha(k))(\frac{1}{2}-\frac{1}{q})}|x-y|^{-(\frac{1}{2}-\frac{1}{q})}\norm{g(u,\cdot)}_{L^{q'}}.\]
Using Minkowski inequality we obtain
\begin{align*}
\norm{TT^*g}_{L^q_{x,t}}\les
2^{k(m(k)-\alpha(k))(\frac{1}{2}-\frac{1}{q})}\normo{|x|^{(\frac{1}{q}-\frac{1}{2})(n-1)}\int{|y|}^{(\frac{1}{q}-\frac{1}{2})(n-1)}\norm{g(y,\cdot)}_{L^{q'}}|x-y|^{-(\frac{1}{2}-\frac{1}{q})}
dy}_{L^q_x}.
\end{align*}
To complete the proof, it suffices to prove
\begin{align}
\aabs{\int_\R \int_\R
\frac{g(y)f(x)}{|x|^{(\frac{1}{2}-\frac{1}{q})(n-1)}
{|y|}^{(\frac{1}{2}-\frac{1}{q})(n-1)}|x-y|^{(\frac{1}{2}-\frac{1}{q})}}dxdy}\les
\norm{g}_{L^{q'}} \norm{f}_{L^{q'}},
\end{align}
which follows immediately from Lemma \ref{lem:DWHLS}, since it is
easy to verify the condition with $q=\frac{4n+2}{2n-1}$,
$\alpha=\beta=(\frac{1}{2}-\frac{1}{q})(n-1)$,
$\lambda=\frac{1}{2}-\frac{1}{q}$, $r=s=q'$, $d=1$. Therefore, we
complete the proof.

\noindent{\bf Step 3.} Sharpness.

It remains to prove that the range of $q$ is optimal. We will prove
that $\norm{e^{it\sqrt{-\Delta}}P_0u_0}_{L^q_{t,x}}\les
\norm{u_0}_2$ fails if $q\leq \frac{2n}{n-1}$, and
$\norm{e^{it\Delta}P_0u_0}_{L^q_{t,x}}\les \norm{u_0}_2$ fails if
$q<\frac{4n+2}{2n-1}$. For the former one, from the proof in step 1
we see that it suffices to disprove: for $q=\frac{2n}{n-1}$
\begin{align}\label{eq:coutwave}
\normo{r^{\frac{n-1}{q}}r^{-\frac{n-1}{2}}\int_{\R}
\psi_0(s)h(s)\cos(rs-(n-1)\pi/4)e^{its} ds}_{L_{t,r\geq 2}^q}\les
\norm{h}_2.
\end{align}
Indeed, by taking $h(s)=1_{[0,10]}(s)$, and from the fact that for
$r\gg 1$
\[
\normo{\int_{\R} \psi_0(s)\cos(rs-\frac{(n-1)\pi}{4})e^{its}
ds}_{L_{|t-r|\leq 1}^q}\ges \norm{
c\hat{\psi_0}(t+r)+\bar{c}\hat{\psi_0}(t-r)}_{L^q_{|t-r|\leq 1}}\ges
1,
\]
we obtain that $\normo{r^{\frac{n-1}{q}}r^{-\frac{n-1}{2}}\int_{\R}
\psi_0(s)h(s)\cos(rs-(n-1)\pi/4)e^{its} ds}_{L_{t,r\geq
2}^q}=\infty$. Thus \eqref{eq:coutwave} fails if $q=\frac{2n}{n-1}$.

To see the latter one, similarly, it suffices to disprove: for
$q<\frac{4n+2}{2n-1}$
\begin{align}\label{eq:coutsch}
\normo{r^{\frac{n-1}{q}}r^{-\frac{n-1}{2}}\int_{\R}
\psi_0(s)h(s)\cos(rs-(n-1)\pi/4)e^{its^2} ds}_{L_{t,r\geq 2}^q}\les
\norm{h}_2.
\end{align}
Indeed, fix a $j$ sufficiently large and take
$h(s)=2^{j/2}1_{|s-1|\les 2^{-j}}$. Then $\norm{h}_2=1$. For $t>0$,
the main contribution of
$r^{\frac{n-1}{q}}r^{-\frac{n-1}{2}}\int_{\R}
h(s)\cos(rs-(n-1)\pi/4)e^{its^2} ds$ is
\begin{align*}
c_nr^{\frac{n-1}{q}}r^{-\frac{n-1}{2}}\int_{\R}
h(s)e^{-irs}e^{its^2}ds=c_n2^{j/2}r^{\frac{n-1}{q}}r^{-\frac{n-1}{2}}\int_{\R}
1_{|s|\leq 2^{-j}}(s)e^{-irs}e^{its^2}e^{i2ts} ds.
\end{align*}
Thus the left-hand side of \eqref{eq:coutsch} is larger than
\[\normo{2^{j/2}r^{\frac{n-1}{q}}r^{-\frac{n-1}{2}}\int_{\R}
1_{|s|\leq 2^{-j}}(s)e^{-irs}e^{its^2}e^{i2ts} ds}_{L^q_{r\sim
2^{2j},|r-2t|\les 2^j}}\ges 2^{j(\frac{2n+1}{q}-\frac{2n-1}{2})}\]
which is unbounded if $q<\frac{4n+2}{2n-1}$. Therefore, we complete
the proof of Theorem \ref{thm:main}.

Next we prove Theorem \ref{thm:mainSch}. First we give the following
maximal function estimates, which generalize the results in
\cite{KPV} for $a\geq 2$ to $a>0$.

\begin{lemma}\label{lem:max}
Assume $a>0$ and $k\geq 0$. Then
\begin{align}\label{eq:lemmax}
\normo{\int_\R
e^{it|\xi|^a}e^{ix\xi}\eta_{0}(\xi/2^k)f(\xi)d\xi}_{L_x^2L_{|t|\les
1}^\infty}\les B(a,k)\norm{f}_2,
\end{align}
where \[B(a,k)=\left\{
\begin{array}{l}
2^{{ak/4}},\quad a\ne 1,\\
2^{{k/2}}, \quad a=1.
\end{array}
\right.\] Moreover, the bounds are sharp.
\end{lemma}

\begin{proof}
By change of variables: $\xi=2^k\eta$ and then $x=2^{-k}y$, we get
that \eqref{eq:lemmax} is equivalent to
\begin{align}\label{eq:lemmaxpf1}
\normo{\int_\R
e^{it|\xi|^a}e^{ix\xi}\eta_{0}(\xi)f(\xi)d\xi}_{L_x^2L_{|t|\les
2^{ka}}^\infty}\les B(a,k)\norm{f}_2.
\end{align}
By $TT^*$ methods, \eqref{eq:lemmaxpf1} is equivalent to
\begin{align}\label{eq:lemmaxpf2}
\normo{\int_{\R^2}\left[\int_\R
e^{i(t-t')|\xi|^a}e^{i(x-x')\xi}\eta_{0}(\xi)d\xi\right]
g(t',x')dt'dx'}_{L_x^2L_{|t|\les 2^{ka}}^\infty}\les
B(a,k)^2\norm{g}_{L_x^2L_{|t|\les 2^{ka}}^1}.
\end{align}
Denote $K_a(x-x',t-t')=\int_\R
e^{i(t-t')|\xi|^a}e^{i(x-x')\xi}\eta_{0}(\xi)d\xi$. By Stationary
phase and Van der Corput lemma, since $|t-t'|\les 2^{ka}$, it is
easy to see that for $a\ne 1$
\[|K_a(x-x',t-t')|\les (1+|x-x'|)^{-1/2}1_{|x-x'|\les 2^{ka}}+|x-x'|^{-4}1_{|x-x'|\gg 2^{ka}}\]
and for $a=1$
\[|K_1(x-x',t-t')|\les 1\cdot 1_{|x-x'|\les 2^{k}}+|x-x'|^{-4}1_{|x-x'|\gg 2^{k}}.\]
Using bounds above and Young's inequality, we get
\begin{align*}
\normo{\int_{\R^2}K_a(x-x',t-t') g(t',x')dt'dx'}_{L_x^2L_{|t|\les
2^{ka}}^1}\les&
\norm{K_a}_{L_x^1L_t^\infty}\norm{g}_{L_x^2L_{|t|\les
2^{ka}}^1}\\
\les& B(a,k)^2\norm{g}_{L_x^2L_{|t|\les 2^{ka}}^1}.
\end{align*}
Thus we obtain the bounds as desired.

It remains to show that the bounds are sharp. First we consider
$a=1$. For $f$ supported in $\{\xi>0\}$, we have
\[L.H.S \mbox{ of } \eqref{eq:lemmaxpf1}\ges \normo{\int_\R
e^{-ix\xi}e^{ix\xi}\eta_{0}(\xi)f(\xi)d\xi}_{L_{|x|\les 2^k}^2}\ges
2^{k/2}\] which shows the sharpness of the bound $2^{k/2}$. Now we
consider $a\ne 1$. Take $f=\theta^{-1/2}1_{|\xi-1|\les \theta}$,
$\theta=2^{-ka/2}$. Then $\norm{f}_2\sim 1$, and
\begin{align*}
L.H.S \mbox{ of } \eqref{eq:lemmaxpf1}\ges&
\theta^{-1/2}\normo{\int_{|\xi|\les \theta}
e^{it(\xi+1)^a}e^{-it}e^{ix\xi}d\xi}_{L_{|x|\les \theta^{-2}}^2L^\infty_{|t|\les 2^{ka}}}\\
\ges& \theta^{-1/2}\normo{\int_{|\xi|\les \theta}
e^{-ix(\xi+1)^a/a}e^{ix/a}e^{ix\xi}d\xi}_{L_{|x|\les
\theta^{-2}}^2}\\
\ges& \theta^{-1/2}=2^{ka/4}
\end{align*}
where in the last inequality we used the fact that
$|(\xi+1)^a-1-a\xi|\les \xi^2$. Thus we complete the proof of the
lemma.
\end{proof}

We present the proof of Theorem \ref{thm:main} in the following two
cases.

{\bf Case 1:} $a\ne 1$.

First we assume that $a\ne 1$. Since \eqref{eq:Stri} is trivial if
$(q,r)=(\infty,2)$, thus by Bernstein's inequality, Riesz-Thorin
interpolation and the classical Strichartz estimates, it suffices to
prove \eqref{eq:Stri} for $(q,r)=(2,r)$, where
$\frac{4n-2}{2n-3}<r<\frac{2n}{n-2}$.

By the scaling transform $(t,x)\to (\lambda^a t, \lambda x)$,
clearly we may assume $k=0$. By the classical Strichartz estimates
(see \cite{KT} for $n\geq 3$ and \cite{Tao2} for $n=2$):
$
\norm{e^{itD^a}P_0f}_{L_t^2L_x^{\frac{2n}{n-2}}}\leq C
\norm{f}_{L_x^2},
$
then we see that from H\"older's inequality, it suffices to prove
\begin{align}\label{eq:Stripf1}
\norm{e^{itD^a}P_0f}_{L_t^2L_{|x|\geq 10}^{r}}\leq C
\norm{f}_{L_x^2}.
\end{align}
As before we divide $u_a(t,|x|)=e^{itD^a}P_0f$ into two parts: the
main term and the error term, namely
\begin{align}\label{eq:decfree}
u_a(t,|x|)=M_a(t,|x|)+E_a(t,|x|)
\end{align}
with
\begin{align*}
{M}_a(t,r)=&c_nr^{-\frac{n-1}{2}}\int_{\R}
\psi_0(s)g(s)s^{\frac{n-1}{2}} e^{i(rs-ts^a)}
ds+\bar{c_n}r^{-\frac{n-1}{2}}\int_{\R} \psi_0(s)
g(s)s^{\frac{n-1}{2}}e^{-i(rs+ts^a)} ds,\\
{E}_a(t,r)=&c_1\int_{\R} \psi_0(s)g(s)s^{n-1}e^{-its^a-
irs}E_{+}(rs)ds-c_2\int_{\R} \psi_0(s)g(s)s^{n-1}e^{-its^a+
irs}E_{-}(rs)ds.
\end{align*}
First we bound the main term. We have
\begin{lemma}
(a) Assume $a\ne 1$, $a>0$, $j\geq 2$ and $2\leq r\leq \infty$. Then
\begin{align}
\|M_a(t,|x|)\|_{L^2_tL_x^r(\R\times A_j)}\les&
2^{j(\frac{2n-1}{2r}-\frac{2n-3}{4})}\norm{f}_{L^2}.\label{eq:Ma}
\end{align}
\end{lemma}
\begin{proof}
For $r=2$, it was proven in Lemma \ref{lem:main}. By Riesz-Thorin
interpolation, it suffices to prove for $r=\infty$. By the
definition of $M_a$ and symmetry, it suffices to show
\begin{align}\label{eq:mainpf1}
2^{-\frac{(n-1)j}{2}}\normo{\int_{\R} \eta(s)g(s) e^{i(rs^{1/a}-ts)}
ds}_{L_t^2L_r^\infty(\R\times I_j)}\les \norm{g}_2
\end{align}
where $\eta(s)$ is a bump function on $\{s\sim 1\}$. By making
change of variables $\xi=s2^{aj}$, $t=2^{aj}x$, we see that it
suffices to prove
\begin{align}\label{eq:mainpf2}
2^{-\frac{(n-1)j}{2}}\normo{\int_{\R} \eta(\xi/2^{aj})g(\xi)
e^{i(t\xi^{1/a}-x\xi)} d\xi}_{L_x^2L_{|t|\leq
2}^\infty}\les2^{-\frac{(2n-3)j}{4}} \norm{g}_2
\end{align}
which reduces to a maximal function estimate associated to the
dispersion $\xi^{1/a}$. Since $a\ne 1$, then \eqref{eq:mainpf2}
follows immediately from Lemma \ref{lem:max}.
\end{proof}

Next, we estimate the error term $E_a(t,|x|)$. This term certainly
has better estimates than the main term, but for our purpose, the
following rough estimates will be enough.
\begin{lemma}
Assume $a\ne 1$, $j\geq 2$ and $2\leq r\leq \frac{2n}{n-2}$. Then
\begin{align}\label{eq:main}
\|E_a(t,|x|)\|_{L^2_tL_x^r(\R\times A_j)}\les
2^{-\frac{j}{2}(\frac{n}{r}-\frac{n-2}{2})}\norm{f}_{L^2}.
\end{align}
\end{lemma}
\begin{proof}
For $r=2$, it was proven in Lemma \ref{lem:error}. For
$r=\frac{2n}{n-2}$ we have
\[\norm{E_a(t,|x|)}_{L_t^2L_x^{\frac{2n}{n-2}}(\R\times
A_j)}\leq
\norm{u_a(t,|x|)}_{L_t^2L_x^{\frac{2n}{n-2}}}+\norm{M_a(t,|x|)}_{L_t^2L_x^{\frac{2n}{n-2}}(\R\times
A_j)}\les \norm{f}_{2}\] where we used the classical endpoint
Strichartz estimates and Lemma \ref{lem:main}.
\end{proof}

We are ready to prove \eqref{eq:Stripf1}. Indeed, since
$\frac{4n-2}{2n-3}<r<\frac{2n}{n-2}$, by Lemma \ref{lem:main} and
Lemma \ref{lem:error}, we can sum over $j\geq 1$:
\[\norm{e^{itD^a}P_0f}_{L_t^2L_{|x|\geq 10}^{r}}\leq \sum_{j=1}^\infty \norm{M_a(t,|x|)}_{L_t^2L_{x}^{r}(\R\times
A_j)}+\sum_{j=1}^\infty \norm{E_a(t,|x|)}_{L_t^2L_{x}^{r}(\R\times
A_j)}\les \norm{f}_2.\]

{\bf Case 2.} $a=1$ and $n\geq 3$.

As in Case 1, it suffices to prove \eqref{eq:Stri} for
$(q,r)=(2,r)$, where $\frac{2n-2}{n-2}<r<\frac{2n-2}{n-3}$. Using
the decomposition \eqref{eq:decfree} and the following lemma, we
immediately obtain \eqref{eq:Stri}.
\begin{lemma}
Assume $j\geq 2$ and $2\leq r\leq \infty$, $2\leq q\leq
\frac{2n-2}{n-3}$. Then
\begin{align*}
\|M_1(t,|x|)\|_{L^2_tL_x^r(\R\times A_j)}\les
2^{j(\frac{n-1}{r}-\frac{n-2}{2})}\norm{f}_{L^2},\quad
\|E_1(t,|x|)\|_{L^2_tL_x^q(\R\times A_j)}\les
2^{-(\frac{n-1}{2q}-\frac{n-3}{4})}\norm{f}_{L^2}.
\end{align*}
\end{lemma}
\begin{proof}
The proof follows exactly as the proof of two Lemmas above, thus we
omit the details.
\end{proof}

Finally, we show the sharpness. $q\geq 2$ is necessary since
\eqref{eq:Stri} is time-translation invariant. The same
counter-example for \eqref{eq:coutsch} shows that
$\frac{2}{q}+\frac{2n-1}{r}\leq n-\frac{1}{2}$ is necessary.

\begin{remark}\label{rem:conj}
We give a remark on Conjecture \ref{conj} for $a=2$. From the proof
of Theorem \ref{thm:mainSch}, we see that to prove \eqref{eq:conj},
it suffices to prove
\[\normo{r^{-1/(2n-1)}\int_\R \psi_0(s)g(s)e^{i(rs-ts^2)}ds}_{L_t^2L_{\{r\geq 1\}}^{\frac{4n-2}{2n-3}}}\les \norm{g}_2.\]
Unfortunately, we are not able to prove this.
\end{remark}

\section{Strichartz estimates in the radial case}

In this section, we will apply Theorem \ref{thm:main} to some
dispersive equations. Since we do not have the decay estimates, then
we use Christ-Kiselev lemma to derive the retarded linear estimates.
First we prove a duality property for radial function.
\begin{lemma}\label{lem:radialdual}
Assume $1\leq p\leq \infty$, $1=1/p+1/p'$, $f\in L^p(\R^n)$ and $f$
is radial. Then
\begin{align}\label{eq:lemradialdual}
\norm{f}_{L^p(\R^n)}=\sup\left\{\bigg|\int_{\R^n}
f(x)g(x)dx\bigg|:g\in L^{p'}(\R^n), g \mbox{ is radial and }
\norm{g}_{L^{p'}}\leq 1\right\}.
\end{align}
\end{lemma}
\begin{proof}
Denote the right-hand side of \eqref{eq:lemradialdual} by $B$. Then
it is obviously that $B\leq \norm{f}_{L^p(\R^n)}$, thus it suffices
to show $\norm{f}_{L^p(\R^n)}\leq B$. By duality, we have
\begin{align*}
\norm{f}_{L^p(\R^n)}=&\sup_{g\in
L^{p'},\norm{g}_{L^{p'}}=1}\left|\int_{\R^n}f(x)g(x)dx\right|\\
=&\sup_{g\in
L^{p'},\norm{g}_{L^{p'}}=1}\left|\int_0^\infty\int_{\mathbb{S}^{n-1}}f(r)g(rx')r^{n-1}drd\sigma(x')\right|\\
=&\sup_{g\in
L^{p'},\norm{g}_{L^{p'}}=1}\left|\int_{\R^n}f(x)\tilde{g}(x)dx\right|,
\end{align*}
where we set
$\tilde{g}(x)=\frac{1}{|\mathbb{S}^{n-1}|}\int_{\mathbb{S}^{n-1}}g(|x|x')d\sigma(x')$.
It's easy to see from H\"older's inequality that $\tilde{g}$ is
radial and $\norm{\tilde{g}}_{L^{p'}}\leq 1$, then we get
$\norm{f}_{L^p(\R^n)}\leq B$ as desired.
\end{proof}

Obviously, Lemma \ref{lem:radialdual} holds similarly for function
$f(t,x)$ spherically symmetric in $x$, e.g. $f\in L_t^pL_x^q$. As a
corollary, we can apply Lemma \ref{lem:radialdual} to get the dual
version estimates of the linear estimates in the radial case.

\begin{lemma}\label{lem:duallinear}
Assume $1\leq q,r\leq \infty$, $1/q+1/q'=1/r+1/r'=1$, $k\in \Z$. If
for all $u_0\in L^2(\R^n)$ and $u_0$ is radial we have
\[\norm{S_\phi(t)P_ku_0}_{L_t^qL_x^r}\les C(k)\norm{u_0}_{L^2},\]
Then for all $f\in L_t^{q'}L_x^{r'}$ and $f$ is spherically
symmetric in space we have
\[\normo{\int_\R S_\phi(-t)[P_kf(t,\cdot)](x)dt}_{L^2(\R^n)}\les C(k)\norm{f}_{L_t^{q'}L_x^{r'}}.\]
\end{lemma}

Christ-Kiselev lemma which was obtained by Christ and Kiselev
\cite{Christ} is very useful in deriving the retarded estimates from
the non-retarded estimates. The one we need is the following, for
its proof we refer the readers to \cite{SmithSogge}.

\begin{lemma}[Christ-Kiselev]
Assume $1\leq p_1,q_1,p_2,q_2\leq \infty$ with $p_1>p_2$. If for all
$f\in L_t^{p_2}L_x^{q_2}$ spherically symmetric in space
\[\normo{\int_\R S_\phi(t-s)(P_kf(s))(x)ds}_{L_t^{p_1}L_x^{q_1}}\les C(k)\norm{f}_{L_t^{p_2}L_x^{q_2}},\]
then we have
\[\normo{\int_0^t S_\phi(t-s)(P_kf(s))(x)ds}_{L_t^{p_1}L_x^{q_1}}\les C(k)\norm{f}_{L_t^{p_2}L_x^{q_2}}\]
holds  with the same bound $C(k)$, for all $f\in L_t^{p_2}L_x^{q_2}$
spherically symmetric in space.

\end{lemma}

Now we are ready to give some new Strichartz estimates for some
concrete equations. First note that from Minkowski inequality and
Littlewood-Paley square function theorem we get if $2\leq
q,r<\infty$ then
\begin{align}\label{eq:sumfre}
\norm{f}_{L_t^qL_x^r}\les \norm{\norm{P_kf}_{L_t^qL_x^r}}_{l_k^2},
\quad \norm{\norm{P_kf}_{L_t^{q'}L_x^{r'}}}_{l_k^2}\les
\norm{f}_{L_t^{q'}L_x^{r'}}.
\end{align}
We will apply \eqref{eq:sumfre} to get the Strichartz estimates on
the whole space.

{\bf 1. Schr\"odinger equation}
\begin{align}\label{eq:Sch}
\left \{
\begin{array}{l}
i\partial_{t}u+\Delta u=F,\quad (t,x)\in \R\times \R^n,\\
u(0)=u_0(x).
\end{array}
\right.
\end{align}
By Duhamel's principle, we get
$u=S(t)u_0-i\int_0^tS(t-\tau)F(\tau)d\tau$, where
$S(t)=e^{-it\Delta}$, which corresponds to $\phi(r)=r^2$. Then we
see that $\phi$ satisfies (H1), (H2), (H3) and (H4) with
$m_1=m_2=\alpha_1=\alpha_2=2$. Thus by Theorem \ref{thm:main} we
obtain for $q\geq\frac{4n+2}{2n-1}$ and if $u_0$ is radial then
\begin{align}\label{eq:Schnew}
\norm{S(t)P_k u_0}_{L^{q}_{t,x}(\R^{n+1})} \les
2^{(\frac{n}{2}-\frac{n+2}{q})k}\norm{u_0}_2.
\end{align}

\begin{definition}\label{def:nDSch}
Suppose $n\geq 2$. The exponent pair $(q,r)$ is said to be n-D
radial Schr\"odinger-admissible if $q,r\geq 2$, and
\begin{align}
\frac{4n+2}{2n-1}\leq q\leq \infty, \frac{2}{q}+\frac{2n-1}{r}\leq
n-\frac{1}{2} \quad or \quad 2\leq q<\frac{4n+2}{2n-1},
\frac{2}{q}+\frac{2n-1}{r}< n-\frac{1}{2}.
\end{align}
\end{definition}

For $n\geq 3$, the n-D radial Schr\"odinger-admissible pairs are
described in the Figure 1 ($a\ne 1$).

\begin{proposition}[Schr\"odinger Strichartz estimate]\label{prop:StriSch}
Suppose $n\geq 2$ and $u,u_0, F$ are spherically symmetric and
satisfy equation \eqref{eq:Sch}. Then
\begin{align}
\norm{u}_{L_t^qL_x^r}+\norm{u}_{C(\R:\dot{H}^\gamma)}\les
\norm{u_0}_{\dot{H}^\gamma}+\norm{F}_{L_t^{\tilde q'}L_x^{\tilde
r'}},
\end{align}
if $\gamma \in \R$, $(q,r)$ and $(\tilde q,\tilde r)$ are both n-D
radial Schr\"odinger-admissible, either $(\tilde q,\tilde r,n)\ne
(2,\infty,2)$ or $(q,r,n)\ne (2,\infty,2)$, and satisfy the ``gap"
condition
\[\frac{2}{q}+\frac{n}{r}=\frac{n}{2}-\gamma,\ \frac{2}{\tilde q}+\frac{n}{\tilde r}=\frac{n}{2}+\gamma.\]
\end{proposition}

\begin{proof}
The case $F=0$ follows from Theorem \ref{thm:mainSch}. Now we assume
$F\ne 0$, $(q,r)$ and $(\tilde q,\tilde r)$ are both n-D radial
Schr\"odinger admissible, $(\tilde q,\tilde r,n)\ne (2,\infty,2)$
and satisfy the ``gap" condition. If $\gamma=0$, this is implied by
the already known estimates \cite{KT}. If $\gamma\ne 0$, then by
scaling it suffices to prove
\begin{align}\label{eq:strschpf2}
\normo{\int_0^tS(t-s)P_0F(s)ds}_{L_t^qL_x^r}\les
\norm{F}_{L_t^{\tilde q'}L_x^{\tilde r'}}.
\end{align}
Since either $q,r>2$ or $\tilde q,\tilde r>2$, then in view of
Christ-Kiselev lemma it suffices to prove
\begin{align}\label{eq:strschpf3}
\normo{\int_\R S(t-s)P_0F(s)ds}_{L_t^qL_x^r}\les
\norm{F}_{L_t^{\tilde q'}L_x^{\tilde r'}},
\end{align}
which follows immediately from the non-retarded linear estimates and
Lemma \ref{lem:duallinear}. Thus we complete the proof of the
proposition.
\end{proof}

\begin{remark}
We remark that we can take $\gamma<0$, which means there are
smoothing effects in the non-retarded Strichartz estimates. This
only holds in the radial case. There are also smoothing effects in
some retarded estimates, but for our purpose, we only derive the
ones without smoothing effect.
\end{remark}

{\bf 2. Wave equation}
\begin{align}\label{eq:wave}
\left \{
\begin{array}{l}
\partial_{tt}u-\Delta u=F,\quad (t,x)\in \R\times \R^n,\\
u(0)=u_0(x),\ u_t(0)=u_1(x).
\end{array}
\right.
\end{align}
By Duhamel's principle, we get
$u=W'(t)u_0+W(t)u_1-\int_0^tW(t-\tau)F(\tau)d\tau$, where
\[W(t)=\frac{\sin(t\sqrt{-\Delta})}{\sqrt{-\Delta}},\quad
W'(t)=\cos(t\sqrt{-\Delta}).\]
This reduces to $W_{\pm}(t):=e^{\pm
it(-\Delta)^{1/2}}$, which corresponds to $\phi(r)=r$. Then we see
that $\phi$ satisfies (H1) and (H2) with $m_1=m_2=1$. Thus by
Theorem \ref{thm:main} we obtain for $q>\frac{2n}{n-1}$ and if $u_0$
is radial then
\begin{align}\label{eq:Schnew}
\norm{W_\pm(t)P_k u_0}_{L^{q}_{t,x}(\R^{n+1})} \les
2^{(\frac{n}{2}-\frac{n+1}{q})k}\norm{u_0}_2.
\end{align}

\begin{definition}
Suppose $n\geq 2$. The exponent pair $(q,r)$ is said to be n-D
radial wave-admissible if $q,r\geq 2$, and one of the following

(1) $n=2$, $(q,r)\in
A_2=\{(q,r):\frac{1}{q}+\frac{1}{r}<\frac{1}{2},q>4\}\cup
\{(4,\infty),(\infty,2)\}$;

(2) $n\geq 3$, $(q,r)\in A_{\geq 3}=\{(q,r):q\geq 2,
\frac{1}{q}+\frac{n-1}{r}<\frac{n-1}{2}\}\cup \{(\infty,2)\}$.
\end{definition}

For $n\geq 4$, the n-D radial wave-admissible pairs are described in
the Figure 1 ($a=1$).

\begin{proposition}[Wave Strichartz estimate]\label{prop:Striwave}
Suppose $n\geq 2$ and $u,u_0,u_1, F$ are spherically symmetric and
satisfy equation \eqref{eq:wave}. Then
\begin{align}
\norm{u}_{L_t^qL_x^r}+\norm{u}_{C([0,T]:\dot{H}^\gamma)}+\norm{\partial_tu}_{C([0,T]:\dot{H}^{\gamma-1})}\les
\norm{u_0}_{\dot{H}^\gamma}+\norm{u_1}_{\dot{H}^{\gamma-1}}+\norm{F}_{L_t^{\tilde
q'}L_x^{\tilde r'}},
\end{align}
if $\gamma \in \R$, $(q,r)$ and $(\tilde q,\tilde r)$ are both n-D
radial wave-admissible, $(\tilde q,\tilde r,n)\ne (2,\infty,3)$, and
satisfy the ``gap" condition
\[\frac{1}{q}+\frac{n}{r}=\frac{n}{2}-\gamma,\ \frac{1}{\tilde q}+\frac{n}{\tilde r}=\frac{n}{2}-1+\gamma.\]
\end{proposition}
\begin{proof}
The proof is similar to that of Proposition \ref{prop:StriSch}. We
omit the details.
\end{proof}

{\bf 3. Klein-Gordon equation}
\begin{align}\label{eq:klein}
\left \{
\begin{array}{l}
\partial_{tt}u-\Delta u+u=F,\\
u(0)=u_0(x),\ u_t(0)=u_1(x).
\end{array}
\right.
\end{align}
By Duhamel's principle, we get
$u=K'(t)u_0+K(t)u_1-\int_0^tK(t-\tau)F(\tau)d\tau$, where
\begin{align*}
K(t)=\omega^{-1} \sin(t\omega),\quad K'(t)=\cos(t\omega), \quad
\omega=\sqrt{I-\Delta}.
\end{align*}
This reduces to the semigroup $K_{\pm}(t):=e^{\pm
it(I-\Delta)^{1/2}}$, which corresponds to $\phi(r)=(1+r^2)^{1/2}$.
By simple calculation,
\[
\phi'(r)=\frac{r}{(1+r^2)^{\half 1}},\quad
\phi''(r)=\frac{1}{(1+r^2)^{\half 3}},
\]
we see that $\phi$ satisfies (H1), (H2), (H3) and (H4) with $m_1=1$,
$\alpha_1=-1$, $m_2=\alpha_2=2$. Thus by Theorem \ref{thm:main} we
obtain for $q\geq \frac{4n+2}{2n-1}$ and if $u_0$ is radial then
\begin{align}
\norm{K_\pm(t)P_k u_0}_{L^{q}_{t,x}(\R^{n+1})} \les
C(q,k)\norm{u_0}_2,
\end{align}
where
\begin{align*}
C(q,k)=
\begin{cases}
2^{(\frac{n}{2}-\frac{n+2}{q})k}, & k\leq 0;\\
2^{(\frac{n}{2}-\frac{n+1}{q})k}, & k\geq 0,
\frac{2n}{n-1}<q\leq \infty;\\
2^{(\frac{n}{2}-\frac{n+1}{q})k+(\frac{1}{2}-\frac{1}{q})k}, & k\geq
0, \frac{4n+2}{2n-1}\leq q\leq \frac{2n}{n-1}.
\end{cases}
\end{align*}

{\bf 4. Beam equation}
\begin{align}\label{eq:Beam}
\left \{
\begin{array}{l}
\partial_{tt}u+\Delta^2 u+u=F,\\
u(0)=u_0(x),\ u_t(0)=u_1(x).
\end{array}
\right.
\end{align}
By Duhamel's principle, we have
$u=B'(t)u_0+B(t)u_1-\int_0^tB(t-\tau)F(\tau)d\tau$, where
\begin{eqnarray*}
B(t)= \omega^{-1}\sin(t\omega), \quad B'(t)=\cos(t \omega), \quad
\omega=\sqrt{I+\Delta^2}.
\end{eqnarray*}
This reduces to the semigroup $B_{\pm}(t):=e^{\pm
it(I+\Delta^2)^{1/2}}$, which corresponding to
$\phi(r)=(1+r^4)^{1/2}$. By simple calculation,
\[
\phi'(r)=2r^3/(1+r^4)^{\half 1}, \quad
\phi''(r)=(6r^2+2r^6)/(1+r^4)^{\half 3},
\]
we know that $\phi$ satisfies (H1) and (H2) with $m_1=\alpha_1=2$,
$m_2=\alpha_2=4$. Thus by Theorem \ref{thm:main} we obtain for
$q\geq\frac{4n+2}{2n-1}$ and if $u_0$ is radial then
\begin{align}
\norm{B_\pm(t)P_k u_0}_{L^{q}_{t,x}(\R^{n+1})} \les
B(q,k)\norm{u_0}_2,
\end{align}
where
\begin{align*}
B(q,k)=\left\{
\begin{array}{l}
2^{(\frac{n}{2}-\frac{n+4}{q})k}, \quad k\leq 0;\\
2^{(\frac{n}{2}-\frac{n+2}{q})k}, \quad k\geq 0.
\end{array}
\right.
\end{align*}

{\bf 5. Fractional-order Schr\"odinger equation}
\begin{align}\label{eq:FSch}
\left \{
\begin{array}{l}
i\partial_{t}u+(-\Delta)^{\frac{\sigma}{2}} u=F,\\
u(0)=u_0(x),
\end{array}
\right.
\end{align}
where $1<\sigma<2$. By Duhamel's principle, we have
$u=S_\sigma(t)u_0+\int_0^tS_\sigma(t-\tau)F(\tau)d\tau$, where
$S_\sigma(t)= e^{-it\phi(\sqrt{-\Delta})}$ with $\phi(r)=r^\sigma$.
By simple calculation, we see that $\phi$ satisfies (H1), (H2), (H3)
and (H4) with $m_1=\alpha_1=m_2=\alpha_2=\sigma$. Thus by Theorem
\ref{thm:main} we obtain for $q\geq\frac{4n+2}{2n-1}$ and if $u_0$
is radial then
\begin{align}
\norm{S_\sigma(t)P_k u_0}_{L^{q}_{t,x}(\R^{n+1})} \les
2^{(\frac{n}{2}-\frac{n+\sigma}{q})k}\norm{u_0}_2.
\end{align}

\begin{proposition}\label{prop:StriFSch}
Suppose $n\geq 2$ and $u,u_0, F$ are spherically symmetric in space
and satisfy equation \eqref{eq:FSch}. Then
\begin{align}
\norm{u}_{L_t^qL_x^r}+\norm{u}_{C(\R:\dot{H}^\gamma)}\les
\norm{u_0}_{\dot{H}^\gamma}+\norm{F}_{L_t^{\tilde q'}L_x^{\tilde
r'}},
\end{align}
if $\gamma \in \R$, $(q,r)$ and $(\tilde q,\tilde r)$ are both n-D
radial Schr\"odinger-admissible (see Definition \ref{def:nDSch}),
$(\tilde q,\tilde r,n)\ne (2,\infty,2)$, and satisfy the ``gap"
condition
\[\frac{\sigma}{q}+\frac{n}{r}=\frac{n}{2}-\gamma,\ \frac{\sigma}{\tilde q}+\frac{n}{\tilde r}=\frac{n}{2}+\gamma.\]
\end{proposition}
\begin{proof}
The proof is similar to that of Proposition \ref{prop:StriSch},
except $(q,r,n)=(2,\infty,2)$. This particular case follows
similarly as for the schr\"odinger equation in \cite{Tao2}. We omit
the details.
\end{proof}

In particular, taking $\gamma=0$ we get a family of Strichartz
estimates without loss of regularity.
\begin{corollary}
Suppose $n\geq 2$, $\frac{2n}{2n-1}<\sigma\leq 2$ and $u,u_0, F$ are
spherically symmetric in space and satisfy equation \eqref{eq:FSch}.
Then
\begin{align}
\norm{u}_{L_t^qL_x^r}+\norm{u}_{C(\R:L^2)}\les
\norm{u_0}_{L^2}+\norm{F}_{L_t^{\tilde q'}L_x^{\tilde r'}},
\end{align}
if $(q,r)$ and $(\tilde q,\tilde r)$ belong to $\{(q,r):q,r\geq
2,\frac{\sigma}{q}+\frac{n}{r}=\frac{n}{2}\}$ and $(\tilde q,\tilde
r,n)\ne (2,\infty,2)$.
\end{corollary}
These estimates without loss of derivatives hold only in the radial
case. Finally we present the Knapp-counterexample to show that the
general non-radial Strichartz estimates have loss of derivative for
$1<\sigma<2$.

Assume that the following inequality hold for general non-radial
function $f$:
\begin{align}\label{eq:Fschg}
\normo{\int_{\R^d}e^{it|\xi|^\sigma}e^{ix\xi}\eta_0(\xi)\hat{f}(\xi)d\xi}_{L^q_tL^r_x}\les
\norm{f}_{L^2}.
\end{align}
Take
\[D=\{\xi=(\xi_1,\xi')\in \R^d:|\xi_1-1|\les \delta,|\xi'|\leq \delta\}.\]
Let $\hat{f}=1_D(\xi)$. Then $\norm{f}_2\sim \delta^{d/2}$, and
\[\int_{\R^d}e^{it|\xi|^\sigma}e^{ix\xi}\eta_0(\xi)\hat{f}(\xi)d\xi=e^{i(t+x_1)}\int_{D}
e^{it(|\xi|^\sigma-\xi_1^\sigma)}e^{it(\xi_1^\sigma-1-\sigma(\xi_1-1))}e^{i(t\sigma+x_1)(\xi_1-1)}e^{ix'\xi'}d\xi.\]
Since in $D$ we have
\[||\xi|^\sigma-\xi_1^\sigma|\les |\xi'|^2\les \delta^2, \, |\xi_1^\sigma-1-\sigma(\xi_1-1)|\les |\xi_1-1|^2\les \delta^2,\]
thus for $|t|\les \delta^{-2},\, |t\sigma+x_1|\les \delta^{-1},\,
|x'|\les \delta^{-1}$, we have
$|\int_{\R^d}e^{it|\xi|^\sigma}e^{ix\xi}\eta_0(\xi)\hat{f}(\xi)d\xi|\sim
|D|$. Therefore, \eqref{eq:Fschg} implies
\[\delta^{-\frac{2}{q}-\frac{d}{r}+\frac{d}{2}}\les 1,\]
which implies immediately that $\frac{2}{q}+\frac{d}{r}\leq
\frac{d}{2}$ by taking $\delta\ll 1$.

\section{Applications to nonlinear equations}\label{sec:nonlinear}

In this section, we apply the improved Strichartz estimates to the
nonlinear equations, e.g. nonlinear Schr\"odinger equation,
nonlinear wave equation. These equations have been studied
extensively.

\subsection{Nonlinear Schr\"odinger equations}

First we consider the semi-linear Schr\"odinger equations:
\begin{align}\label{eq:NLS}
i\partial_tu+\Delta u= \mu |u|^{p}u,\ \ u(0)=u_0(x),
\end{align}
where $u(t,x):\R\times\R^n\rightarrow\C,\ n\geq 2$, $u_0\in
\dot{H}^s$, $p>0$, $\mu=\pm 1$. It is easy to see that equation
\eqref{eq:NLS} is invariant under the following scaling transform:
for $\lambda>0$
\[u(t,x)\rightarrow \lambda^{2/p} u(\lambda^2t,\lambda x),\ u_0(x)\rightarrow \lambda^{2/p} u_0(\lambda x).\]
Then the space $\dot{H}^{s_{sch}}$, where
\[s_{sch}=\frac{n}{2}-\frac{2}{p},\]
is the critical space to \eqref{eq:NLS} in the sense of scaling,
namely, $\norm{\lambda^{2/p} u_0(\lambda
\cdot)}_{\dot{H}^{s_{sch}}}=\norm{u_0}_{\dot{H}^{s_{sch}}}$. In
particular, if $p<4/n$, then $s_{sch}<0$, which is our main concern.

The well-posedness and scattering for the nonlinear Schr\"odinger
equation \eqref{eq:NLS} were extensively studied. We refer the
readers to \cite{CaWe,Bour,I-team1,KeMe1,Visan,Dodson2d,Dodson3d}
and the reference therein. It is well-known that the threshold of
$\dot{H}^s$-wellposedness for \eqref{eq:NLS} is $s\geq
\max(0,s_{Sch})$. However, in the radial case we prove the following

\begin{theorem}\label{thm:Schwp}
Assume $n\geq 2$, $0<p<4/n$, $s_{sch}=\frac{n}{2}-\frac{2}{p}$,
$\frac{1-n}{2n+1}\leq s_{sch}<0$, and $u_0$ is radial. Then we have

(1) Small data scattering: If $\norm{u_0}_{\dot{H}^{s_{sch}}}\leq
\delta$ for some $\delta\ll 1$, then there exist a unique global
solution $u$ to \eqref{eq:NLS} such that
\[u\in C(\R:\dot{H}^{s_{sch}})\cap L_{t,x}^{\frac{p(n+2)}{2}}(\R\times \R^n),\]
and $u_{\pm}\in \dot{H}^{s_{sch}}$ such that
$\norm{u-S(t)u_{\pm}}_{\dot{H}^{s_{sch}}}\rightarrow 0$, as
$t\rightarrow \pm \infty$.

(2) Large data local well-posedness: If $u_0\in \dot{H}^s$ for some
$s_{sch}\leq s<0$, then there exists $T>0$ and a unique solution
$u\in C((-T,T):\dot{H}^{s})\cap
L_{t,x}^{\frac{2(n+2)}{n-2s}}((-T,T)\times \R^n)$.
\end{theorem}

\begin{proof}
The proof is quite standard. The main point is to choose the
resolution space. By Duhamel's principle, we have
\[
u=\Phi_{u_0}(u)=S(t)u_0 + \mu \int_0^t S(t-s)
(|u|^{\frac{4}{n-2s_{sch}}}u)(s) ds.
\]
First, we show (1). Take\footnote{Indeed, the choice of index was
determined by a group of linear equation or inequalities. The choice
is not unique, and we choose the simple one here. We will remark
more on this for the wave equation.}
\[q=r= \frac{2(n+2)}{n-2s_{sch}}, \quad \tilde{q}=\tilde{r}=
\frac{2(n+2)}{n+2s_{sch}}.\] It is easy to verify that
$(q,r),(\tilde q,\tilde r)$ satisfy the conditions in Proposition
\ref{prop:StriSch} with $\gamma=s_{sch}$. Thus by applying
Proposition \ref{prop:StriSch}, we get
\begin{align*}
\|\Phi_{u_0}(u)\|_{L^q_{t,x}} +
\|D^{s_{sch}}\Phi_{u_0}(u)\|_{L^\infty_{t}L^{2}_x} \les &
\norm{S(t)u_0}_{L_{t}^qL_x^q} + \||u|^{\frac{4}{n-2s_{sch}}}u\|_{L^{\tilde{q}'}_{t,x}}\\
\les & \|D^{s_{sch}} u_0\|_{L^2} +
\|u\|^{1+\frac{4}{n-2s_{sch}}}_{L^{\frac{(n-2s_{sch}+4)\tilde{q}'}{n-2s_{sch}}}_{t,x}}.
\end{align*}
Note that $\tilde{q}'= \frac{2(n+2)}{n-2s_{sch}+4}$, then
$\frac{(n-2s_{sch}+4)\tilde{q}'}{n-2s_{sch}} = q$. Thus part (1)
follows from standard fixed point arguments (\cite{CaWe}).

Next, we show part (2). Local well-posedness for equation
\eqref{eq:NLS} in $\dot{H}^{s_{sch}}$ follows from the fact that for
$q=\frac{2(n+2)}{n-2s_{sch}}<\infty$
\[\lim_{T\rightarrow 0}\norm{S(t)u_0}_{L_{t\in
[-T,T]}^qL_x^q}=0.\] Now we assume $s_{sch}<s<0$. Take
$q=r=\frac{2(n+2)}{n-2s}$ and
\[\frac{1}{\tilde q}=\frac{n+2s}{2n+4}-\frac{2n(s-s_{sch})}{(n+2)(n-2s_{sch})}, \ \frac{1}{\tilde r}=\frac{n+2s}{2n+4}+\frac{4s-4s_{sch}}{(n+2)(n-2s_{sch})}.\]
It is easy to verify that $(q,r),(\tilde q,\tilde r)$ satisfy the
conditions in Proposition \ref{prop:StriSch} with $\gamma=s$, and
$(p+1)\tilde r'= q$. Thus by applying Proposition
\ref{prop:StriSch}, we get for some $\theta>0$
\begin{align*}
\|\Phi_{u_0}(u)\|_{L^q_{t,x}} +
\|D^{s}\Phi_{u_0}(u)\|_{L^\infty_{t}L^{2}_x}
\les & \|D^{s} u_0\|_{L^2} + \||u|^{\frac{4}{n-2s_{sch}}}u\|_{L_{t\in [-T,T]}^{\tilde{q}'}L_x^{\tilde r'}}\\
\les & \|D^{s} u_0\|_{L^2} +T^{\theta}
\|u\|^{1+\frac{4}{n-2s_c}}_{L^{\frac{(n-2s_{sch}+4)\tilde{r}'}{n-2s_{sch}}}_{t,x}}.
\end{align*}
Thus part (2) also follows from standard fixed-point argument.
\end{proof}

\begin{remark}\rm
In part (2) of Theorem \ref{thm:Schwp}, the existence time $T$
depends only on $\norm{u_0}_{\dot{H}^s}$ for $s>s_{sch}$, but on the
profile of $u_0$ for $s=s_{sch}$.

Actually, we can obtain more conclusions than Theorem
\ref{thm:Schwp}. Using the similar proof, we can obtain if $s_{sch}<
\frac{1-n}{2n+1}$, namely $0<p< \frac{8n+4}{2n^2+3n-2}$, large data
local well-posedness for \eqref{eq:NLS} hold in $\dot{H}^{s}$ for
$s>s_{1}$ with
\begin{align}
s_1=
\begin{cases}
\frac{1-n}{2n+1}, & \frac{2}{n}\leq p < \frac{8n+4}{2n^2+3n-2};\\
\frac{np-n^2p}{2(-1 + 2 n + n p)}, & p\leq \frac{2}{n}.
\end{cases}
\end{align}
Actually, $s_0$ is determined by the following groups of linear
equations:
\begin{align*}
\left\{
\begin{array}{l}
2\leq q,r,\tilde q,\tilde r\leq \infty,\\
\frac{2}{q}+\frac{2n-1}{r}=n-\frac{1}{2},\\
\frac{2}{q}+\frac{n}{r}=\frac{n}{2}-\gamma,\\
\frac{2}{\tilde
q}+\frac{n}{\tilde r}=\frac{n}{2}+\gamma,\\
(p+1)\tilde r'=r, \tilde q=\infty.
\end{array}
\right.
\end{align*}
Then we can also obtain $(q,r),(\tilde q,\tilde r)$ for $s>s_0$,
which can be used to prove local well-posedness as in the proof of
Theorem \ref{thm:Schwp}.

\begin{figure}[htbh]
\begin{center}
\includegraphics{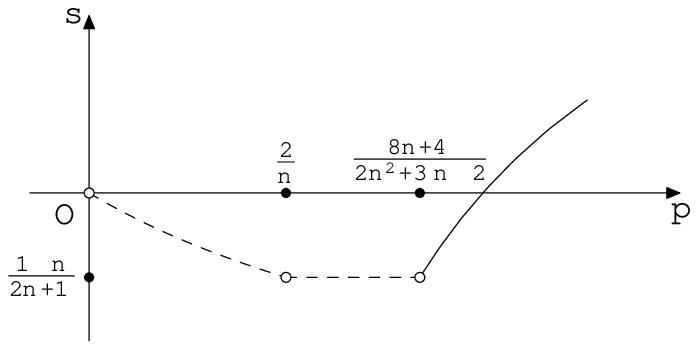}
\end{center}
\caption{ \mbox{$\dot{H}^s$ well-posedness for NLS}}\label{hfig1}
\end{figure}

The same conclusions obtained above certainly hold for general
nonlinear terms $F(u)$, for example, if $F$ satisfies
\begin{align}\label{eq:nonlinearF}
|F(u)|\les& |u|^{p+1},\nonumber\\
|u||F'(u)|\sim& |F(u)|.
\end{align}
We describe the regularity $s$ for $\dot{H}^s$ local well-posedness
and nonlinear increasing rate $p+1$ in Figure 3.
\end{remark}

\begin{remark}\label{rem:inho}
Part (2) in Theorem \ref{thm:Schwp} also holds for data $u_0\in
H^s$. Indeed, we simply construct the resolution space as following
\[\norm{u}_{Y_T}=\norm{P_{\leq 0}u}_{L_{[-T,T]}^\infty L^2}+\norm{P_{\geq 1}u}_{L^{q}_{|t|\leq T,x}}.\]
\end{remark}

\subsection{Nonlinear wave equations}
Next, we consider the semi-linear wave equations:
\begin{align}\label{eq:NLW}
\left \{
\begin{array}{l}
\partial_{tt}u-\Delta u=\mu|u|^pu,\quad (t,x)\in \R\times \R^n,\\
u(0)=u_0(x),\ u_t(0)=u_1(x).
\end{array}
\right.
\end{align} where
$u(t,x):\R\times\R^n\rightarrow\R,\ n\geq 2$, $\mu=\pm 1$, $u_0\in
\dot{H}^s$, $u_1\in \dot{H}^{s-1}$. It is easy to see that equation
\eqref{eq:NLW} is invariant under the following scaling transform:
for $\lambda>0$
\[u(t,x)\rightarrow \lambda^{2/p} u(\lambda t,\lambda x),\ u_0(x)\rightarrow \lambda^{2/p} u_0(\lambda x), u_1(x)\rightarrow \lambda^{(2+p)/p} u_1(\lambda x).\]
Then the space $\dot{H}^{s_{w}}\times \dot{H}^{s_{w}-1}$, where
\[s_{w}=\frac{n}{2}-\frac{2}{p},\]
is the critical space to \eqref{eq:NLW} in the sense of scaling,
namely, $\norm{\lambda^{2/p} u_0(\lambda
\cdot)}_{\dot{H}^{s_{w}}}=\norm{u_0}_{\dot{H}^{s_{w}}}$.

The well-posedness and scattering for equation \eqref{eq:NLW} were
deeply studied. We refer the readers to
\cite{GSV,LindbladSogge,SStruwe,Sogge,Gril1,Gril2,SStruwe2,SStruwe3,SStruwe4,KT,Tao,KeMe2}
and the reference therein. In this section, we study the
well-posedness theory for \eqref{eq:NLW} in $\dot{H}^s\times
\dot{H}^{s-1}$ with radial initial data. As was indicated in the
introduction, the sharp results at the critical regularity were
obtained in \cite{LindbladSogge} if $s_w\geq 1/2$. Thus we restrict
ourselves to the case $s_w<1/2$, and we find an threshold $s_0(n)$
for the critical GWP in the radial case:
\begin{align}\label{eq:s0n}
s_0(n)=
\begin{cases}
\frac{5-\sqrt{17}}{4}, &n=2,\\
\frac{12-\sqrt{129}}{6}, &n=3,\\
\frac{n^2+3n-3-\sqrt{n^4+6n^3-n^2-14n+9}}{4n-4},&n\geq 4.
\end{cases}
\end{align}
It seems that this is the optimal regularity by our methods. We
prove the following
\begin{theorem}\label{thm:Wavewp}
Assume $n\geq 2$, $0<p<\frac{4}{n-1}$,
$s_{w}=\frac{n}{2}-\frac{2}{p}$, $s_0(n)<s_{w}<1/2$ with $s_0(n)$
given by \eqref{eq:s0n}, and $u_0$ is radial. Then

(1) If
$\norm{u_0}_{\dot{H}^{s_{w}}}+\norm{u_1}_{\dot{H}^{s_w-1}}\leq
\delta$ for some $\delta\ll 1$, then there exists a unique global
solution $u$ to \eqref{eq:NLW} such that \[u\in
C(\R:\dot{H}^{s_{w}})\cap C^1(\R:\dot{H}^{s_w-1})\cap
L_t^qL_x^r(\R\times \R^n),\]where $(q,r)$ are given in the proof,
and $(u_{\pm},v_\pm)\in \dot{H}^{s_{w}}\times \dot{H}^{s_w-1}$ such
that
\[\norm{u-W'(t)u_{\pm}}_{\dot{H}^{s_w}}+\norm{u_t-W(t)v_{\pm}}_{\dot{H}^{s_w-1}}\rightarrow 0, \mbox{ as }
t\rightarrow \pm \infty.\]

(2) If $u_0\in \dot{H}^s$ for some $s_{w}\leq s<1/2$, then there
exists $T>0$ and a unique solution $u$ to \eqref{eq:NLW} defined on
$(-T,T)$ such that
\[u\in
C((-T,T):\dot{H}^{s})\cap C^1((-T,T):\dot{H}^{s-1})\cap
L_t^qL_x^r((-T,T)\times \R^n),\] where $(q,r)$ is the index given by
part (1) for $s_w=s$.
\end{theorem}

\begin{proof}[Proof of Theorem \ref{thm:Wavewp}]
By Duhamel's principle, we have
\[
u=\Phi_{u_0,u_1}(u)=W'(t)u_0+W(t)u_1+\mu\int_0^t W(t-s)
(|u|^{\frac{4}{n-2s_{w}}}u)(s) ds.
\]
First we show part (1) and explain how $s_0$ is obtained. The main
issue is to choose the admissible pairs $(q,r),(\tilde q,\tilde r)$
so that the contraction argument is closed\footnote{The ideas for
the Schr\"odinger equations are the same. However, the choice of the
index for the wave equations is more complicated.}. By the choice of
$(q,r)$ and $(\tilde q,\tilde r)$, we should have
\begin{align*}
\|\Phi_{u_0,u_1}(u)\|_{L_t^qL_x^r}
\les & \norm{W'(t)u_0}_{L_t^qL_x^r}+\norm{W(t)u_1}_{L_t^qL_x^r} + \||u|^{\frac{4}{n-2s_{w}}}u\|_{L_t^{\tilde{q}'}L_x^{\tilde r'}}\\
\les & \|D^{s_{w}} u_0\|_{L^2}+\|D^{s_{w}-1} u_1\|_{L^2}+
\|u\|^{1+\frac{4}{n-2s_{w}}}_{L_t^qL_x^r}.
\end{align*}
The inequalities above hold if $(q,r),(\tilde q,\tilde r)$ satisfy
\begin{align}\label{eq:waveglobalcdt}
\left\{
\begin{array}{l}
(q,r),(\tilde q,\tilde r)\mbox{ is n-D radial wave admissible},\\
\frac{1}{q}+\frac{n}{r}=\frac{n}{2}-s_w,\\
\frac{1}{\tilde
q}+\frac{n}{\tilde r}=\frac{n}{2}-1+s_w,\\
(p+1)\tilde r'=r,\ (p+1)\tilde q'=q.
\end{array}
\right.
\end{align}
Therefore, once we find solution to \eqref{eq:waveglobalcdt}, then
part (1) follows from standard arguments. We give a solution to
\eqref{eq:waveglobalcdt} case by case:

Case 1: $\frac{1}{2n}<s_{w}\leq 1/2$.
\[(q,r)=(\frac{2n+2}{n-2s_{w}},\frac{2n+2}{n-2s_{w}}),\, (\tilde{q},\tilde{r})=
(\frac{2n+2}{n+2s_{w}-2},\frac{2n+2}{n+2s_{w}-2}).\]

Case 2: $s_0<s_{w}\leq \frac{1}{2n}$.

Case 2a: $n=2$.
\[(q,r)=(\frac{3-s_w}{(1-s_w)^2},\frac{3-s_w}{1-s_w}),\ (\tilde{q},\tilde{r})=
(\frac{1}{s_w},\infty).\]

Case 2b: $n=3$. For some $0<\theta\ll 1$,
\[(\frac{1}{q},\frac{1}{r})=(2s_w-3\theta,\frac{1}{2}-s_w+\theta),\ (\tilde{q},\tilde{r})=
(\frac{q}{q-p-1},\frac{r}{r-p-1}).\]

Case 2c: $n\geq 4$.
\begin{align*}
(q,r)=(\frac{2n+8-4s_{w}}{n-2s_{w}},
\frac{2n^2+8n-4ns_{w}}{n^2+3n-4ns_{w}+4s_{w}^2-6s_{w}}),\,
(\tilde{q},\tilde r)=(2,\frac{2n}{n+2s_{w}-3}).
\end{align*}
Therefore, part (1) is proved.

Next we show part (2). Local well-posedness in $\dot{H}^{s_{w}}$
follows from the fact that for the choice of $(q,r)$ in the proof of
part (1)
\[\lim_{T\rightarrow 0}\norm{W'(t)u_0}_{L_{t\in
[-T,T]}^qL_x^r}+\norm{W(t)u_1}_{L_{t\in [-T,T]}^qL_x^r}=0.\] Now we
assume $s_{w}<s<1/2$. The proof is very similar to the Schr\"odinger
equations. We take $(q,r)$ to be the one corresponding to $s$ in
part (1), and then take $(\tilde q,\tilde r)$ to close the argument.
We omit the details.
\end{proof}

\begin{remark}
As the Schr\"odinger equation, if $s_{w}\leq s_0(n)$, namely $p\leq
\frac{4}{n-2s_0(n)}$, we can't prove well-posedness in
$\dot{H}^s\times \dot{H}^{s-1}$ down to $s=s_{w}$. However, we can
also improve the well-posedness results in \cite{LindbladSogge}. We
only mention the case $n\geq 4$, we obtain if $\frac{3}{n}<p\leq
\frac{4}{n-2s_0(n)}$, then large data local well-posedness hold in
$\dot{H}^s\times \dot{H}^{s-1}$ for $s>s_2$ with
\[s_2=\frac{np-3}{2np+2n-2}.\]
Indeed, take $\tilde q=2,\tilde r=\frac{2n}{n-3+2s}$, and $(q,r)$
such that
\begin{align*}
\frac{1}{q}=\frac{n}{2}-\frac{n}{r}-s,\
\frac{1}{r}=\frac{1}{p+1}-\frac{1}{(p+1)\tilde r}.
\end{align*}
Then by this choice we can prove the local well-posedness using the
similar arguments as the proof of Theorem \ref{thm:Wavewp}.

The same results hold for general nonlinear terms $F(u)$, e.g. $F$
satisfying \eqref{eq:nonlinearF}. We describe the regularity $s$ for
$\dot{H}^s\times \dot{H}^{s-1}$ local well-posedness and nonlinear
increasing rate $p+1$ for \eqref{eq:NLW} in Figure 4.

\begin{figure}[htbh]
\begin{center}
\includegraphics{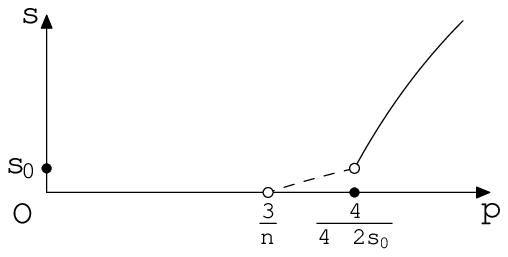}
\end{center}
\caption{\mbox{$\dot{H}^s\times \dot{H}^{s-1}$ well-posedenss for
NLW}}\label{hfig1}
\end{figure}
\end{remark}

\subsection{Nonlinear fractional-order Schr\"odinger equation}

In this section, we apply the improved Strichartz estimates to the
nonlinear fractional-order Schr\"odinger equation:
\begin{align}\label{eq:FNLS}
i\partial_tu+(\sqrt{-\Delta})^{\sigma} u= \mu |u|^{p}u,\ \
u(0)=u_0(x),
\end{align}
where $u(t,x):\R\times\R^n\rightarrow\C,\ n\geq 2$, $1<\sigma<2$,
$\mu=\pm 1$, $u_0\in \dot{H}^s$. To the best of our knowledge, there
are few results concerning the well-posedness for \eqref{eq:FNLS}.
The main reason is that the usual Strichartz estimates derived by
the decay estimates have a loss in derivatives except the trivial
one $L_t^\infty L_x^2$. Then one may need to use other methods, for
example, local smoothing effect methods, and using of the $X^{s,b}$
space. These methods will certainly be able to provide some results
at least when $p$ is an even integer.

However, in the radial case, we obtain more Strichartz estimates for
\eqref{eq:FNLS}, some of which don't have a loss in derivative. Then
our idea is to use these kinds of estimates. The equation
\eqref{eq:FNLS} has the following two symmetries which we will use.
One is the scaling invariance: for any $\lambda>0$, \eqref{eq:FNLS}
is invariant under the following transformation
\[u(t,x)\rightarrow \lambda^{\sigma/p} u(\lambda^\sigma t,\lambda x),\ u_0(x)\rightarrow \lambda^{\sigma/p} u_0(\lambda x).\]
The others are the conservation laws: if $u$ is smooth solution to
\eqref{eq:FNLS}, then
\begin{align*}
\frac{d}{dt}\int_{\R^n}|u|^2dx=&0,\quad (mass)\\
\frac{d}{dt}\int_{\R^n}||\nabla|^{\sigma/2}
u|^2-\frac{\mu}{p+2}|u|^{p+2}dx=&0. \quad (energy)
\end{align*}
Then we see the space $\dot{H}^{s_c}$, where
\[s_c=\frac{n}{2}-\frac{\sigma}{p}\]
is critical in the sense of scaling, and $\mu=-1$ is the defocusing
case while $\mu=1$ corresponds to the focusing case. We will use the
following lemma:

\begin{lemma}[Fractional chain rule, \cite{Christ2}]\label{lem:Fchain}
Suppose $G\in C^1(\C)$, $s\in (0,1]$, and $1<p,p_1,p_2<\infty$ are
such that $\frac{1}{p}=\frac{1}{p_1}+\frac{1}{p_2}$. Then
\[\norm{|\nabla|^sG(u)}_p\les \norm{G'(u)}_{p_1}\norm{|\nabla|^su}_{p_2}.\]
\end{lemma}

In view of the conservation laws, we only consider the nonlinear
terms between mass-critical to energy-critical, namely,
$\frac{2\sigma}{n}\leq p\leq\frac{2\sigma}{n-\sigma}$. First we
consider the critical $\dot H^s$ well-posedness theory of
\eqref{eq:FNLS}. For the simplicity of notation, we denote
$S_\sigma(t)= e^{it(\sqrt{-\Delta})^{\sigma}}$. We prove the
following

\begin{theorem}\label{thm:FSchwp}
Assume $n\geq 2$, $\frac{2n}{2n-1}\leq \sigma<2$, $p\ge
\frac{2\sigma}{n}$ $s_c=\frac{n}{2}-\frac{\sigma}{p}$, and $u_0\in
H^{s_c}$ is radial. Then the IVP \eqref{eq:FNLS} admits

(1) Small data scattering: If $\norm{u_0}_{\dot{H}^{s_c}}\leq
\delta$ for some $\delta\ll 1$, then there exists a unique global
solution \[u\in C(\R:{H}^{s_c})\cap
L_t^{p+2}L_x^{\frac{2n(p+2)}{2(n-\sigma)+np}}(\R\times \R^n),\] and
$u_{\pm}\in \dot{H}^{s_c}$ such that
$\norm{u-S_\sigma(t)u_{\pm}}_{\dot{H}^{s_c}}\rightarrow 0$, as
$t\rightarrow \pm \infty$.

(2) Large data local well-posedness: There exists $T=T(u_0)>0$ and a
unique solution $u\in C((-T,T):{H}^{s_c})\cap
L_t^{p+2}L_x^{\frac{2n(p+2)}{2(n-\sigma)+np}}((-T,T)\times \R^n)$.
\end{theorem}

\begin{proof}
Since $\sigma \geq \frac{2n}{2n-1}$, then $\frac{2(n+\sigma)}{n}
\geq \frac{2(2n+1)}{2n-1}$. Thus it is easy to see that
$(2+\frac{2\sigma }{n}, 2+\frac{2\sigma }{n})$ is an n-D radial
Schr\"odinger admissible pair and then by Proposition
\ref{prop:StriFSch} we get
\[\norm{S_\sigma(t)u_0}_{L_{t,x}^{2+\frac{2\sigma
}{n}}(\R\times \R^n)}\les \norm{u_0}_{L_x^2}.\] Then interpolating
this with the trivial one $\norm{S_\sigma(t)u_0}_{L_t^\infty
L_x^2(\R\times \R^n)}\les \norm{u_0}_{L_x^2}$, we get more
estimates. The key point is that these Strichartz estimates are
without loss of regularity.

With these estimates, the proof is quite standard, for example see
\cite{Visan}. First we show part (1). By Duhamel's principle, we
have
\[
u=\Phi_{u_0}(u)=S_\sigma(t) u_0 + \mu \int_0^t S_\sigma (t-s)
(|u|^{p}u)(s) ds,
\]
Take
\[q=\tilde q=p+2, \ r=\tilde r=\frac{2n(p+2)}{2(n-\sigma)+np}.\]
It is easy to verify that $(q,r),(\tilde q,\tilde r)$ satisfy the
conditions in Proposition \ref{prop:StriFSch} with $\gamma=0$. Then
we define the set $X= B_1\cap B_2$ endowed with the metric $d(u,v):=
\|u-v\|_{L^q_tL^r_x}$, where
\begin{align*}
B_1=&\{u\in L^\infty_tH^{s_c}_x(\R\times \R^n): \
\|u\|_{L^\infty_tH^{s_c}_x} \le 2\|u_0\|_{H^{s_c}_x}
+C(n)(2\eta)^{1+p}\},\\
 B_2=&\{u\in L^{q}_t W^{s_c,r}_x(\R\times
\R^n): \ \|u\|_{L^{p+2}_t \dot W^{s_c,r}_x} \le 2\eta,\,\,
\|u\|_{L^{q}_t L^{r}_x} \le 2C(n)\|u_0\|_{L^2_x}\},
\end{align*}
with some sufficient small $\eta>0$ to be determined latter. It's
easy to see that $(X,d)$ is complete and we will show that the
solution map $\Phi_{u_0}$ is a contraction on $(X,d)$ with the
initial data condition
\begin{align}
\norm{u_0}_{\dot{H}^{s_c}}\leq \eta\ll 1.
\end{align}

First we show $\Phi_{u_0}:X\to X$. Since $q'=\frac{p+2}{p+1}, \
r'=\frac{2n(p+2)}{2(n+\sigma)+np}$, then it is easy to see that
\[\frac{1}{q'}=\frac{1}{q}+\frac{1}{pq},\quad \frac{1}{r'}=\frac{1}{r}+\frac{2\sigma}{n(p+2)}.\]
Then by Proposition \ref{prop:StriFSch}, fractional chain rule Lemma
\ref{lem:Fchain} and Sobolev embedding, we find that for $u\in X$,
\begin{align*}
\|\Phi_{u_0}(u)\|_{L^\infty_tH^{s_c}_x(I\times \R^n)}\le &
\|u_0\|_{H^{s_c}_x} + C(n) \|\langle\nabla\rangle^{s_c}
(|u|^p)u\|_{L^{q'}_tL^{r'}_x}\\
\le & \|u_0\|_{H^{s_c}_x} + C(n) \|\langle\nabla\rangle^{s_c}
u\|_{L^{q}_tL^{r}_x}\| u\|^p_{L^{q}_tL^{\frac{np(p+2)}{
2\sigma}}_x}\\
\le & \|u_0\|_{H^{s_c}_x} + C(n) (2\eta + 2C(n) \|u_0\|_{L^2_x}
)\||\nabla|^{s_c}u\|^p_{L^{q}_t L^{r}_x}\\
\le & \|u_0\|_{H^{s_c}_x} + C(n) (2\eta + 2C(n) \|u_0\|_{L^2_x}
)(2\eta)^p
\end{align*}
and similarly,
\begin{align*}
\|\Phi_{u_0}(u)\|_{L^q_tL^{r}_x}\le & C(d)\|u_0\|_{L^2_x} + C(d) \|(|u|^p)u\|_{L^{q'}_tL^{r'}_x}\\
\le & C(d)\|u_0\|_{L^2_x} + 2C(d)^2 \|u_0\|_{L^{2}_x}(2\eta)^{p},
\end{align*}
and
\begin{align*}
\||\nabla|^{s_c}\Phi_{u_0}(u)\|_{L^q_tL^{r}_x} \le
&\||\nabla|^{s_c}S_\sigma(t)u_0\|_{L^q_tL^{r}_x}
+ C(n) (2\eta)^{p+1}\\
\le & C(n)\eta+ C(n) (2\eta)^{p+1}.
\end{align*}
Thus, choosing $\eta_0=\eta_0(n)$ sufficiently small, we see that
for $0<\eta\leq \eta_0$, the functional $\Phi_{u_0}$ maps the set
$X$ back to itself. To see that $\Phi_{u_0}$ is a contraction, we
repeat the computations above and get for $u,v\in X$
\begin{align*}
\|\Phi_{u_0}(u)-\Phi_{u_0}(v)\|_{L^q_tL^{r}_x}\le & C(d) \|(|u|^p)u - (|v|^p)v\|_{L^{q'}_tL^{r'}_x}\\
\le & C(d) (2\eta)^{p} \|u-v\|_{L^{q}_tL^r_x}.
\end{align*}
Thus for $\eta$ sufficiently, the map $\Phi_{u_0}$ is a contraction.
By the contraction mapping theorem, it follows that $\Phi_{u_0}$ has
a fixed point in $X$. The rest of part (1) (e.g. the uniqueness)
follows from standard arguments \cite{Visan}.

Next, to show part (2), we see that since $q\neq \infty$, then
\[
\lim_{T\rightarrow 0}\norm{|\nabla|^{s_c} S_\sigma(t)u_0}_{L_{t\in
[-T,T]}^qL_x^r}=0.
\]
Then part (2) follows from standard fixed-point argument too.
\end{proof}

Using the similar arguments above, and in view of the conservation
laws, it is not difficult to prove the following corollary for which
we do not give the proof.

\begin{corollary}[$H^s$ subcritical]
Assume $n\geq 2$, $\frac{2n}{2n-1}<\sigma<2$ and $u_0$ is radial.
Then for $0<p<\frac{2\sigma}{n}$, the IVP \eqref{eq:FNLS} is
globally well-posed if $u_0\in L^2$; and for $\frac{2\sigma}{n}\leq
p<\frac{2\sigma}{n-2\sigma}$, the IVP \eqref{eq:FNLS} is locally
well-posed (globally well-posed in the defocusing case) if $u_0\in
H^{\sigma/2}$.
\end{corollary}

Indeed, we can prove some other subtle well-posedness results. We
can also go below $L^2$, as long as $\sigma$ is close to $2$.
However, we do not pursue this. On the other hand, in the
$H^s$-critical case, we assumed $u_0\in H^{s_c}$ instead of $u_0\in
\dot{H}^{s_c}$ as in the work of Cazenave and Weissler \cite{CaWe}.
This makes the proof much simpler \cite{Visan}. We will address this
in our consequent works which will concern the large data scattering
theory for \eqref{eq:FNLS}.

\noindent{\bf Acknowledgment.} The authors are very grateful to
Victor Lie for helpful discussion. The endpoint radial Strichartz
estimate for the Schr\"odinger equation is out of our discussion.
The authors also thank Prof. Kenji Nakanishi for precious comments
and suggestions. This material is based upon work supported by the
National Science Foundation under agreement No. DMS-0635607 and The
S. S. Chern Fund. Any opinions, findings and conclusions or
recommendations expressed in this material are those of the authors
and do not necessarily reflect the views of the National Science
Foundation or The S. S. Chern Fund.


\footnotesize

\begin{thebibliography}{99}
\bibitem{Bour}{J. Bourgain, Global well-posedness of defocusing critical nonlinear Schr\"odinger equation in the radial case,
 J. Amer. Math. Soc. 12 (1999), 253--283.}
 \bibitem{BKPSV}{B. Birnir, C. Kenig, G. Ponce, N. Svanstedt, L. Vega, On the ill-posedness of the IVP for the generalized
 Korteweg-de Vries and nonlinear Schr\"odinger equations, J. London Math. Soc. 53 (1996),551--559.}
\bibitem{CaWe}{T. Cazenave, F. B. Weissler, The Cauchy problem for the critical nonlinear Schr\"odinger
equation in $H^s$, Nonlinear Analysis, 14 (1990), 807--836.}
\bibitem{Christ}{M. Christ, A. Kiselev, Maximal functions associated to filtrations, J. Funct. Anal. 179 (2001) 406--425.}
\bibitem{Christ2}{M. Christ, M. Weinstein, Dispersion of small amplitude solutions of the generalized
Korteweg-de Vries equation, J. Funct. Anal. 100 (1991), 87--109.}
\bibitem{CCT}{M. Christ, J. Colliander, T. Tao, Ill-posedness for nonlinear Schr\"odinger and wave equations, arxiv:math.AP/0311048.}
\bibitem{I-team1}{J. Colliander, M. Keel, G. Staffilani, H. Takaoka, T. Tao, Global well-posedness and scattering for the energy critical
nonlinear Schr\"odinger equations in $\R^3$, Annals of Mathematics,
167 (2008), 767--865.}
\bibitem{Dodson3d}{B. Dodson, Global well-posedness and scattering for the defocusing, $L^2$-critical, nonlinear
Schr\"odinger equation when $d\geq 3$, arXiv:0912.2467v1.}
\bibitem{Dodson2d}{B. Dodson, Global well-posedness and scattering for the defocusing, $L^2$-critical,
nonlinear Schr\"odinger equation when $d=2$, arXiv:0912.2467v1.}

\bibitem{FangWang}{D. Fang, C. Wang, Weighted Strichartz Estimates with Angular Regularity and their
Applications, arXiv:0802.0058v2.}
\bibitem{FangWang2}{D. Fang, C. Wang, Some Remarks on Strichartz Estimates for Homogeneous Wave Equation, Nonlinear Anal. 65 (2006), no. 3, 697--706.}
\bibitem{GSV}{J. Ginibre, A. Soffer and G. Velo, The global Cauchy problem for the critical
nonlinear wave equation, J. Funct. Anal. 110 (1992), 96--130.}
\bibitem{Gril1}{M. Grillakis, Regularity and asymptotic behaviour of the wave equation with a
critical nonlinearity, Ann. of Math. 132 (1990), 485--509. }
\bibitem{Gril2}{M. Grillakis, Regularity for the wave equation with a critical nonlinearity,
Comm. Pure Appl. Math. 45 (1992), 749--774.}
\bibitem{GPW}{Z. Guo, L. Peng, B. Wang, Decay estimates for a class of wave equations, Journal of Functional Analysis, 254 (2008), 1642--1660.}
\bibitem{Hidano1}{K. Hidono, Nonlinear Schr\"odinger equations with radially
symmetric data of critical regularity. Funkcial. Ekvac. 51 (2008),
no. 1, 135--147.}
\bibitem{Hidano2}{K. Hidano, Small solutions to semi-linear wave equations
with radial data of critical regularity. Rev. Mat. Iberoam. 25
(2009), no. 2, 693--708.}
\bibitem{HiKu}{K. Hidano, Y. Kurokawa. Weighted HLS inequalities for radial functions and Strichartz estimates for wave and Schr?dinger equations. Illinois J. Math., 52 (2008), 365--388.}
\bibitem{JWY}{J.-C. Jiang, C. Wang, X. Yu, Generalized and Weighted Strichartz Estimates, arXiv:1008.5397.}
\bibitem{KT}{M. Keel, T. Tao, Endpoint Strichartz estimates, Amer. J. Math., {120} (1998), 360--413.}
\bibitem{KPV}{C. E. Kenig, G. Ponce and L. Vega, Well-posedness of the initial value problem
for the Korteweg-de Vries equation, J. Amer. Math. Soc., 4 (1991)
323--347.}
\bibitem{KeMe1}{C. E. Kenig, F. Merle, Global well-posedness, scattering and blow up for the energy-critical, focusing, nonlinear Schr\"odinger equation in the radial case,
Invent. Math. 166 (2006), 645--675.}
\bibitem{KeMe2}{C. E. Kenig, F. Merle, Global well-posedness, scattering and blow up for the energy-critical, focusing
nonlinear wave equation in the radial case, Acta Math. 201 (2008),
no. 2, 147--212.}
\bibitem{KM}{S. Klainerman, M. Machedon. Space-time Estimates for Null Forms and the Local Existence Theorem. Comm. Pure Appl. Math, 46 (1993), 1221--1268. }
\bibitem{Visan}{R. Killip, M. Visan, Nonlinear Schr\"odinger equations at critical regularity, Clay lecture notes.}
\bibitem{LindbladSogge}{H. Lindblad, C. D. Sogge, On existence and scattering with minimal regularity for semilinear wave equations, Journal of Functional Analysis, 130  (1995), 357--426.}
\bibitem{Shao}S. Shao, Sharp linear and bilinear restriction estimates for
paraboloids in the cylindrically symmetric case, Revista
Matem\'atica Iberoamericana, 25 (2009), no. 3, 1127--1168.
\bibitem{Shao2}S. Shao, a note on the cone restriction conjecture in the
cylindrically symmetric case, Proc. Amer. Math. Soc., 137 (2009),
135--143.
\bibitem{SStruwe}{J. Shatah and M. Struwe, Well-posedness in the energy space for semilinear
wave equations with critical growth, Internat. Math. Res. Notices 7
(1994), 303--309.}
\bibitem{SStruwe2}{J. Shatah, M. Struwe, Regularity results for nonlinear wave equations, Ann.
of Math. 138 (1993), 503--518.}
\bibitem{SStruwe3}{J. Shatah, M. Struwe, Well-posedness in the energy space for semilinear
wave equations with critical growth, Internat. Math. Res. Notices 7
(1994), 303--309.}
\bibitem{SStruwe4}{J. Shatah, M. Struwe, Geometric wave equations, Courant Lecture Notes
in Mathematics, 1998.}
\bibitem{SmithSogge}{H. F. Smith, C. D. Sogge, Global Strichartz estimates for nontrapping perturbations of Laplacian, Comm.
Partial Differential Equations, 25 (2000), 2171--2183.}
\bibitem{Sogge}{C. Sogge, ¡°Lectures on nonlinear wave equations,¡± Monographs in Analysis II,
International Press, 1995.}
\bibitem{Stein1971}
E. M. Stein, G. Weiss, Introduction to {F}ourier analysis on
{E}uclidean spaces, Princeton University Press, 1971.
\bibitem{Stein1993}
E. M. Stein  Harmonic analysis: real-variable methods,
orthogonality, and oscillatory integrals, Princeton University
Press, 1993.
\bibitem{Stein1958}{E. M. Stein and G. Weiss, Fractional integrals in n-dimensional Euclidean space, J. Math.
Mech., 7 (1958).}
\bibitem{Sterbenz}{J. Sterbenz, Angular regularity and Strichartz estimates for the wave equation, With an
appendix by Igor Rodnianski, Int. Math. Res. Not. 2005 No.4,
187--231.}
\bibitem{STR}{R. S. Strichartz, Restrictions of Fourier transforms to quadratic surfaces
 and decay of solutions of wave equation,  Duke Math. J. {44}  (1977), 705--714.}
\bibitem{Tao}{T. Tao, Low regularity semi-linear wave equations, Comm. Partial Differential Equations, 24 (1999), 599--629.}
\bibitem{Tao2}{T. Tao. Spherically averaged endpoint Strichartz estimates for the two-
dimensional Schr\"odinger equation. Comm. Partial Differential
Equations, 25 (2000), 1471--1485.}


\end{thebibliography}
\end{document}